\documentclass[11pt]{article}
\usepackage{changes}
\usepackage{amsmath,amssymb,amsthm}
\usepackage{bm}
\usepackage{makeidx}  
\usepackage{graphicx} 
\usepackage{color}
\usepackage{url}
\usepackage{bm}
\usepackage{enumerate}
\usepackage{listings}
\usepackage[normalem]{ulem}

\newenvironment{polynomial}
{\par\vspace{\abovedisplayskip}%
	\setlength{\leftskip}{\parindent}%
	\setlength{\rightskip}{\leftskip}%
	\medmuskip=4mu plus 2mu minus 2mu
	\binoppenalty=0
	\noindent$\displaystyle}
{$\par\vspace{\belowdisplayskip}}

\usepackage{amsfonts}

\def\cB{{\mathcal B}}

\def\cM{{\mathcal{MM}}}
\def\cGM{{\mathcal{GMM}}}

\def\cM{{\mathcal {M}}}

\def\F{{\mathbb F}}

\def\F{{\mathbb F}}

\def\00{{\bf 0}}
\def\+{\oplus}

\def\\{\cr}
\def\({\left(}
\def\){\right)}

\providecommand{\newoperator}[3]{%
  \newcommand*{#1}{\mathop{#2}#3}}

\newoperator{\FD}{\mathrm{FD}}{\nolimits}
\newcommand{\EQ}{\begin{equation}}
\newcommand{\EN}{\end{equation}}

\def\whitebox{{\hbox{\hskip 1pt
        \vrule height 6pt depth 1.5pt
        \lower 1.5pt\vbox to 7.5pt{\hrule width
                  3.2pt\vfill\hrule width 3.2pt}%
        \vrule height 6pt depth 1.5pt
        \hskip 1pt } }}
\def\qed{\ifhmode\allowbreak\else\nobreak\fi\hfill\quad\nobreak\whitebox\medbreak}

\theoremstyle{plain}
\newtheorem{theo}{Theorem}[section]
\newtheorem{cor}[theo]{Corollary}
\newtheorem{lemma}[theo]{Lemma}
\newtheorem{prop}[theo]{Proposition}

\newtheorem{ex}[theo]{Example}

\theoremstyle{definition}
\newtheorem{rem}[theo]{Remark}
\newtheorem{defi}[theo]{Definition}

\newtheorem{op}{Open Problem}

\usepackage{algpseudocode}
\usepackage[Algorithmus,section]{algorithm}
\algnewcommand{\IfOneRow}[1]{\State\algorithmicif\ #1,}
\algnewcommand{\EndifOneRow}{}

\algdef{SE}[SUBALG]{Indent}{EndIndent}{}{\algorithmicend\ }%
\algtext*{Indent}
\algtext*{EndIndent}

\usepackage{caption}
\makeatletter
\renewcommand{\ALG@name}{Algorithm}
\makeatother

\numberwithin{theo}{section}
\numberwithin{equation}{section}
\numberwithin{table}{section}
\numberwithin{figure}{section}

\usepackage{booktabs}
\usepackage[colorlinks=true,citecolor=blue,linkcolor=blue,urlcolor=blue,bookmarks,bookmarksopen,bookmarksdepth=3,backref=page]{hyperref}
\usepackage{bookmark}
\bookmarksetup{
	numbered, 
	open,
}

\renewcommand*{\backref}[1]{}
\renewcommand*{\backrefalt}[4]{%
	\ifcase #1 (Not cited.)%
	\or        (Cited on page~#2.)%
	\else      (Cited on pages~#2.)%
	\fi}

\date{}

\title{Almost Maiorana-McFarland bent functions}
\author{Sadmir Kudin \footnote{University of Primorska, FAMNIT \&
        IAM, Koper, Slovenia, e-mail: sadmir.kudin@iam.upr.si}
    \and  Enes Pasalic \footnote{University of Primorska, FAMNIT \&
        IAM, Koper, Slovenia, e-mail: enes.pasalic6@gmail.com}
    \and Alexandr Polujan \footnote{Otto-von-Guericke University, Magdeburg, Germany, e-mail: alexandr.polujan@gmail.com}
\and Fengrong
Zhang \footnote{School of Cyber Engineering,  Xidian University, Xian, 710071, P.R. China, e-mail:zhfl203@163.com}
\and Haixia Zhao \footnote{Key Laboratory of Cognitive Radio and Information Processing, Ministry of Education, and also with School of Mathematics and Computational Science,
	Guilin University of Electronic Technology,
	Guilin 541004, China, e-mail: guetzhx@163.com}
}

\begin{document}
\maketitle              

\begin{abstract}
	 In this article, we study bent functions on $ \mathbb{F}_2^{2m} $ of the form $ f(x,y) = x \cdot \phi(y) + h(y) $, where $ x \in \mathbb{F}_2^{m-1} $ and $ y \in \mathbb{F}_2^{m+1} $, which form the generalized Maiorana-McFarland class (denoted by $ \mathcal{GMM}_{m+1} $) and are referred to as almost Maiorana-McFarland bent functions. We provide a complete characterization of the bent property for such functions and determine their duals. Specifically, we show that $f$ is bent if and only if the mapping $ \phi $ partitions $ \mathbb{F}_2^{m+1} $ into 2-dimensional affine subspaces, on each of which the function $ h $ has odd weight. While the partition of $\mathbb{F}_2^{m+1} $ into 2-dimensional affine subspaces is crucial for the bentness, we demonstrate that the algebraic structure of these subspaces plays an even greater role in ensuring that the constructed bent functions $f$ are excluded from the completed Maiorana-McFarland class $ \mathcal{M}^\# $ (the set of bent functions that are extended-affine equivalent to the Maiorana-McFarland class $\mathcal{M}$). Consequently, we investigate which properties of mappings $ \phi \colon \mathbb{F}_2^{m+1} \to \mathbb{F}_2^{m-1} $ lead to bent functions of the form $ f(x,y) = x \cdot \phi(y) + h(y) $ both inside and outside $ \mathcal{M}^\# $ and provide construction methods for suitable Boolean functions $ h $ on $ \mathbb{F}_2^{m+1} $. As part of this framework, we present a simple algorithm for constructing partitions of the vector space $ \mathbb{F}_2^{m+1} $ together with appropriate Boolean functions $ h $ that generate bent functions outside $ \mathcal{M}^\# $. When $ 2m = 8 $, we explicitly identify many such partitions that produce at least $ 2^{78} $ distinct bent functions on $ \mathbb{F}_2^8 $ that do not belong to $ \mathcal{M}^\# $, thereby generating more bent functions outside $ \mathcal{M}^\# $ than the total number of 8-variable bent functions in $ \mathcal{M}^\# $ (whose cardinality is approximately $2^{77} $). Additionally, we demonstrate that concatenating four almost Maiorana-McFarland bent functions outside $ \mathcal{M}^\# $, i.e., defining $ f = f_1 || f_2 || f_3 || f_4 $ where $ f_i \notin \mathcal{M}^\# $, can result in a bent function $ f \in \mathcal{M}^\# $. This finding essentially answers an open problem posed recently in Kudin et al. (IEEE Trans. Inf. Theory 71(5): 3999–4011, 2025). Conversely, using a similar approach to concatenate four functions $ f_1 || f_2 || f_3 || f_4 $, where each $ f_i \in \mathcal{M}^\# $, we generate bent functions that are provably outside $ \mathcal{M}^\# $.
\end{abstract}


\section{Introduction}\label{sec:intr}

Bent functions form a special class of Boolean functions in an even number of variables, characterized by a flat Walsh spectrum. These functions are notable for their rich combinatorial properties and applications in cryptography. For example, their Hamming distance to the set of affine functions is maximal. A Boolean function is bent if and only if its support forms a Hadamard difference set~\cite{Dillon1974}. Moreover, the Cayley graphs constructed from bent functions are strongly regular~\cite{BernasconiC1999}. The notion of bent functions was introduced by Rothaus in the mid sixties, and in 1976  the same author provided~\cite{Rothaus1976}  one of the first secondary methods of constructing new bent functions from the known ones.

The construction methods of bent functions can, in general,  be divided into two categories:
 primary (direct) construction methods  and secondary constructions. The former approach does not  exploit some known bent functions in the design and it  rather uses a suitable set of affine functions (typical for the Maiorana-McFarland method \cite{McFarland1973}) or a collection of disjoint $n/2$-dimensional subspaces to construct a bent function on $\F_2^n$ (typical for the partial spread class introduced by Dillon \cite{Dillon1974}). These direct approaches also give rise to so-called primary classes of bent functions. The known primary classes of bent functions include the Partial Spread ($\mathcal{PS}$) class, introduced by Dillon \cite{Dillon1974}, and the Maiorana-McFarland ($\mathcal{M}$) class, first presented in \cite{McFarland1973}.
For  an exhaustive survey on bent functions the interested reader is
referred to \cite{CarlMes2016} and the recent monograph on Boolean functions by Carlet \cite{Carlet2021}. 

The secondary constructions employ some initial bent functions along with a set of conditions (which are easier or harder to satisfy) to generate new bent functions. Carlet \cite{CC93} introduced two new classes of bent functions, the so-called $\mathcal{C}$ and $\mathcal{D}$, by suitably modifying functions from the $\mathcal{M}$ class, and showed that a subclass $\mathcal{D}_0$ contains instances that are outside the completed classes $\mathcal{PS}^\#$ and $\mathcal{M}^\#$. Even though certain efforts of specifying more bent functions outside $\cM^\#$ has been made e.g. in \cite{RothEnes2016,OutsideMM,BFAExtended,Polujan2020,DAM2020,Bent_Decomp2022,ZPBBInfComp,PPKZ_BFA23_CCDS},
 a major progress has been achieved recently in \cite{PPKZ2023,ISIT-IEEE2024} where large families of bent functions outside $\cM^\#$ were obtained by concatenating suitable bent functions in the Maiorana-McFarland class. 
One of the main challenges in this research topic is to provide some efficient secondary constructions of bent functions and to possibly provide a  classification of these functions. The latter problem is evident from the work of Langevin and Leander \cite{Langevin}, showing that the number of bent functions in eight variables that belong to the main two primary classes is only a small fraction (about the size of $2^{76}$ up to adding affine terms)  of the totality of approximately $2^{106}$ bent functions in eight variables.
 The purpose of this article is to provide explicit design of bent functions outside $\cM^\#$ of very large cardinality. For instance, when $n=8$ we can explicitly define at least $2^{78}$ such functions, which is larger than the cardinality of all bent functions in $\cM^\#$, being  approximately $2^{77}$ (this recent estimation was obtained independently in~\cite{LangevinPolujan2024BFA} and~\cite{KolomeecMMFExtensions2025}, and aims to improve the original estimation of approximately $2^{81}$ given in~\cite{Langevin}).

 For the above mentioned purpose, we will consider the generalized Maiorana-McFarland ($\cGM_{n/2+k}$) class of Boolean functions in $n$ variables, with $n$ being even, which is a broader class of functions compared to the $\cM$ class  and  can be viewed  as a concatenation of affine functions in $n/2-k$ variables with $k>0$,  whereas typically  bent functions in  the $\cM$ class correspond to a concatenation of affine functions in $n/2$ variables.
 The class  $\mathcal{GMM}_{n/2+k}$  was introduced in \cite{Camion1991} and was initially used in the design of resilient functions  \cite{Camion1991,Carlet_Indirect,WGZhang2014}.
  More precisely, the members of $\mathcal{GMM}_{n/2+k}$  are defined as $f(x,y)=x\cdot\phi(y) + 
 h(y), $ for $(x,y)   \in \F_2^{\frac{n}{2}-k} \times \F_2^{\frac{n}{2}+k}$, where $\phi \colon  \F_2^{\frac{n}{2}+k} \rightarrow  \F_2^{\frac{n}{2}-k}$ and $h$ is a Boolean function (notice that $f$ is  an affine function in $n/2-k$ variables for any fixed $y$).

In this article, denoting $n=2m$, we provide a full characterization of bent functions of the form $f(x,y) =x \cdot \phi(y) + h(y)$, where $x \in \F_2^{m-1}$, $y \in \F_2^{m+1}$ and $m \geq 4$ (as otherwise these bent functions would belong to $\cM^\#$), that stem from the $\mathcal{GMM}_{\frac{n}{2}+1}$ class. We remark that necessary and sufficient  conditions for bentness of the functions in $\mathcal{GMM}_{\frac{n}{2}+k}$  were obtained in \cite[Theorem 6.40]{Logachev2012}, for all even $n \geq 4$ and  $1 \leq k \leq \frac{n}{2}-1$. However, the structure of the mapping $\phi$ and the associated preimages was not specified. When $k=1$, we show that the preimages of $\phi$ (which is necessarily 4-to-1) are necessarily 2-dimensional affine subspaces and for a fixed  $\phi$  the exact number of different functions $h$ so that $f(x,y)=x \cdot \phi(y) + h(y)$ is bent is given. Their dual bent functions are also specified though  a compact algebraic description is not easily obtained. It turns out that a particular class of non-trivial decompositions of the vector space $\F_2^{m+1}$ into 2-dimensional affine subspaces, referred to as proper decompositions (see Definition~\ref{def: properdec}), is of crucial importance, since any  partition 
that satisfies the condition given in Proposition~\ref{prop: trivialpartition}  only leads to bent functions in $\cM^\#$. Nonetheless, Example~\ref{ex:necessity} demonstrates that the condition in Proposition~\ref{prop: trivialpartition}, while sufficient, is not necessary for inclusion in~$\cM^\#$.
Using a simple algorithm for partitioning the vector space $\F_2^{5}$ (thus for $m=4$) into  2-dimensional affine subspaces we could confirm that out of  4960 such partitions 3785 of them are proper, i.e., do not satisfy the criterion in Proposition~\ref{prop: trivialpartition}. Moreover, it was verified that all the obtained proper partitions actually define bent functions outside $\cM^\#$, which then gives a total cardinality of at least $2^{78}$  bent  functions outside $\cM^\#$ in 8 variables (computing the orders of the automorphism groups of EA-inequivalent classes and applying the orbit-stabilizer result, as outlined in \cite[Theorem 2]{LangevinPolujan2024BFA}). 
In another direction, we also demonstrate that even when each bent function $f_i$, $i=1,2,3,4$, is outside $\cM^\#$, the concatenation $f=f_1||f_2||f_3||f_4$ can still result in a bent function $f \in \cM^\#$.
This result essentially answers an open problem recently posed  in \cite{ISIT-IEEE2024}. In this context, even a stronger statement regarding the inclusion property is valid, as given in Theorem \ref{th:outsidein}. More precisely, every function in the $\cGM_{m+1}$ class (even those outside $\cM^{\#}$) is a part of a function in the $\cM^{\#}$ class in a larger number of variables. On the other hand,  using a similar approach of defining $f=f_1||f_2||f_3||f_4$,  we show that selecting  suitable $f_i \in \cM^\#$ generates bent functions that are provably outside $\cM^\#$, see Theorem \ref{th:4concatoutsideMM}.

Although large families of bent functions outside $\cM^\#$ can be designed using proper partitions of $\F_2^{m+1}$  into 2-dimensional affine subspaces, an explicit algebraic specification of the mapping  $\phi \colon \F_2^{m+1} \to \F_2^{m-1}$ remains unclear. In this context, the extended \eqref{eq: P1} property, introduced in \cite{PPKZ2023} for permutations over $\F_2^m$ that satisfy $D_aD_b \phi \neq 0_m$, called \eqref{eq: P1ext} in this article, appears to be of similar importance for mappings $\phi \colon \F_2^{m+1} \to \F_2^{m-1}$. 
Even though the property of being outside $\cM^\#$ is not necessarily induced by the property \eqref{eq: P1ext}, when the class $\cGM_{m+k}$ is considered, for $k \geq 1$, it turns out that the property \eqref{eq: P1ext} of $\phi \colon \F_2^{m+k} \to \F_2^{m-k}$ is sufficient  to state the exclusion from $\cM^\#$, see Theorem~\ref{th:uniqueMsubspace}. We propose sufficient conditions for these $2^{2k}$-to-1 mappings $\phi$ to satisfy \eqref{eq: P1ext}, but  further investigations are needed (see Open problem \ref{op:overallconditions}) since $\phi$ additionally must partition $\F_2^{m+k}$ in  a non-trivial manner (see also Example \ref{ex:P1extnotbent}). 
Finally, we show the connection between the $\mathcal{GMM}_{\frac{n}{2}+1}$ class whose elements are $f(x,y)=x \cdot \phi(y) + h(y)$, where $x \in \F_2^{m-1}$, $y \in \F_2^{m+1}$,  and the strict Maiorana-McFarland class which contains bent functions of the form  $f(x,y)=x \cdot \pi(y) + h'(y)$, with $x, y \in \F_2^{m}$ and $\pi$ is a permutation of $\F_2^{m}$. 
 
 This paper is organized as follows. In Section \ref{sec:prel}, we recall some  basic facts concerning Boolean and in particular bent functions.
 In Section \ref{sec:conc}, we define the  $\mathcal{GMM}_{\frac{n}{2}+k}$ class 
and recall the conditions that ensure  bentness of $f \in \cGM_{\frac{n}{2}+k}$. Both necessary and sufficient conditions that a function $f(x,y)=x \cdot \phi(y) + h(y) \in \cGM_{\frac{n}{2}+1}$ is bent, along with a precise structure of $\phi$, are given in Section \ref{sec:bentness}. In this section, we also specify exactly the dual bent function $f^*$ of $f$ and determine the exact number of functions $h$ so that $f$ remains bent for a fixed mapping $\phi$. Moreover, we provide a solution to an open problem in \cite{ISIT-IEEE2024} by showing that $f=f_1||f_2||f_3||f_4$ may still belong to $\cM^\#$ even though $f_i \not \in \cM^\#$. In Section \ref{sec:inclusion}, we provide a sufficient condition for $\phi: \F_2^{m+1} \to \F_2^{m-1}$ and the corresponding partition of $\F_2^{m+1}$ so that $f(x,y)=x \cdot \phi(y)+h(y)$ is a bent function in $\cM^\#$. By not fulfilling this criterion, thus providing proper partitions of $\F_2^{m+1}$, we could specify a large set of bent functions, of cardinality at least $2^{78}$, that are outside $\cM^\#$. Furthermore, we also propose an explicit method of specifying $f=f_1||f_2||f_3||f_4 \not \in \cM^\#$, where $f_i \in \cM^\#$. In Section \ref{sec:uniqueMandP1}, we consider algebraic properties of $\phi$ (by introducing the extended \eqref{eq: P1ext} property) and relate them to the exclusion from $\cM^\#$. We provide certain sufficient conditions for  $\phi$ to satisfy \eqref{eq: P1ext}, but  further investigation is needed since $\phi$ additionally must partition $\F_2^{m+k}$ in  a proper manner to achieve bentness of $f(x,y)=x\cdot \phi(y) + h(y)$.
Some concluding remarks are given in Section \ref{sec:concl}.

\section{Preliminaries}\label{sec:prel}
 Let $\F_2$
  denote the binary field of characteristic two, and accordingly let $\F_{2^n}$ denote the Galois field of order $2^n$ whose corresponding vector space is $\F_2^n$ (once the basis of $\F_{2^n}$ over $\F_2$ is fixed).
 Any function from $\mathbb{F}_{2}^n$  to
$\mathbb{F}_2$ (alternatively from $\mathbb{F}_{2^n}$ to $\F_2$) is called an $n$-variable Boolean function, and the
set of all Boolean functions in $n$ variables is denoted by
$\cB_n $. A Boolean function  $f\colon\F_2^n \rightarrow \F_2$ is commonly represented as a multivariate polynomial $f(x_1, \ldots,
x_n)$ over $\F_2$ called the {\em algebraic normal form} (ANF) of $f$.
 More precisely, $f(x_1, \ldots, x_n)$ can be written as
\begin {eqnarray}
f(x_1,\dots,x_n)=\sum_{u\in \F_2^n}\lambda_u\left(\prod_{i=1}^n
x_i^{u_i}\right),
\end {eqnarray} for $\lambda_u\in \F_2,
u=(u_1,\ldots,u_n) \in \F_2^n$. We use $0_n$ to denote the all-zero vector in $\mathbb{F}^n_2$. For a function $f \in \mathcal{B}_n$ and $a \in \F_2^n$, we denote by $f^a $  the function in $ \mathcal{B}_n$ defined by
$f^a(x) = f(x+a)$, for all $x\in\F_2^n$.

The Hamming weight of $x=(x_1,\ldots,x_n)\in \mathbb{F}^n_2$,  denoted by $\operatorname{wt}(x)$, is the cardinality of the set $\{ i \in \{ 1, \ldots ,n \} \mid x_i=1 \}$.
The \emph{Walsh-Hadamard transform} of $f\in\mathcal{B}_n$, at any point $\omega\in\mathbb{F}^n_2$, is defined as an integer valued function $
W_{f}(\omega)=\sum_{x\in \mathbb{F}_2^n}(-1)^{f(x)+ \omega\cdot x}$ (the summation is taken over integers), where $``\cdot"$ denotes the standard inner (dot)
product of two vectors, that is, $\omega \cdot x=\omega_1x_1+ \cdots
+ \omega_nx_n$ (the sum is taken modulo two).
In general, we employ ``+'' instead of more precise $\oplus$ and its use should be clear from the context. Sometimes, to avoid confusion, we will anyway use  the correct symbol $\oplus$ for the binary addition.

A function $f\in\cB_n $, where $n$ is even, is called bent if
$W_f(\omega)=\pm 2^{\frac{n}{2}}$, for every $\omega \in \mathbb{F}_2^n$.  Furthermore, the \emph{dual function} of a bent function $f\in \mathcal{B}_n$, denoted by $f^*$,  is the Boolean function in $n$ variables defined by $W_f(w)=2^{n/2}(-1)^{f^*(w)}$,  which  is also a bent function.
 The \textit{first-order derivative} of a function $f\in\mathcal{B}_n$ in the direction $a \in \F_2^n$ is the mapping $D_{a}f(x)=f(x+a) +  f(x)$. An element $a \in \F_2^n \setminus \{0_n\}$ is called a (nonzero) {\em linear structure} of $f$ if $D_af(x) =c \in \F_2$, for all $x \in \F_2^n$. Derivatives of higher orders are defined recursively, i.e., the \emph{$k$-th-order derivative} of a function $f\in\mathcal{B}_n$ is defined by $D_Vf(x)=D_{a_k}D_{a_{k-1}}\ldots D_{a_1}f(x)=D_{a_k}(D_{a_{k-1}}\ldots D_{a_1}f)(x)$, where $V=\langle a_1,\ldots,a_k \rangle$ is a vector subspace of $\F_2^n$ spanned by vectors  $a_1,\ldots,a_k\in\F_2^n$ that form a basis of $V$. Throughout the paper, we use the notation $\langle S, a\rangle$ to denote the linear span of the vectors from the set $S \cup \{a\}$, where $S \subset \F_2^n$ and $a \in \F_2^n$.

    The (strict) \textit{Maiorana-McFarland class} $\cM$ is the set of $n$-variable ($n=2m$) Boolean bent functions of the form
 \[
 f(x,y)=x \cdot \pi(y)+ h(y), \mbox{ for all } x, y\in\F_2^m,
 \]
 where $\pi$ is a permutation on $\F_2^m$, and $h$ is an arbitrary Boolean function on
 $\F_2^m$. We recall that the {\em completed} class is obtained by applying the so-called extended affine equivalence to the functions in a given class. More precisely, if we consider the class $\mathcal{M}$,  given an arbitrary  $f \in \mathcal{M}$ defined on $\F_2^n$, this affine equivalence class includes a set of functions $\{g\}$ obtained by
 \begin{equation}\label{eq:extendedAE}
 g(x)=f(xA+b) + c \cdot x + d,
 \end{equation}
 where $A \in GL(n,\F_2)$ (the group of invertible $n\times n$ matrices over $\F_2$ under multiplication), $b,c \in \F_2^n$ and $d \in \F_2$.
 Thus, the  completed class  $\mathcal{M}^\#$ can be defined as 
 
 \begin{equation} \label{eq:completed}
 \mathcal{M}^\#=\{f(xA+b) + c \cdot x + d: f \in \mathcal{M}, A \in GL(n,\F_2), b,c \in \F_2^n,  d \in \F_2\}.
 \end{equation}
 Using Dillon's criterion, (stated as Lemma~\ref{lem M-M second} below; for a proof, see~\cite{Carlet2021} or \cite[p. 19]{Polujan2020}), one can show that a given Boolean bent function $f\in\mathcal{B}_n$ does (not) belong to the completed Maiorana-McFarland class $\mathcal{M}^\#$.
 \begin{lemma} \cite[p. 102]{Dillon1974}\label{lem M-M second}
    Let $n=2m$. A Boolean bent function $f\in\mathcal{B}_n$ belongs to $\cM^{\#}$ if and only if
    there exists an $m$-dimensional vector subspace $U$ of $\F_2^n$ such that the second-order derivatives
    $ D_{a}D_{b}f(x)=f(x) +  f(x +  a) +  f(x +  b) +  f(x +  a +  b)$
    vanish for any $ a,  b \in U$.
 \end{lemma}
	
	Following the terminology of~\cite{Polujan2020}, we introduce the concept of $\mathcal{M}$-subspaces for Boolean functions (not necessarily bent) as follows.  
	
	\begin{defi}  
		Let $f \in \mathcal{B}_n$ be a Boolean function. A vector subspace $V \subseteq \F_2^n$ is called an $\mathcal{M}$-subspace of $f$ if the second-order derivative satisfies  
		\[
		D_{a}D_{b}f = 0, \quad \text{for all } a, b \in V.
		\]  
	\end{defi}  
	Throughout the article, for shortness, we use the notation $D_aD_bf$ instead of $D_aD_bf(x)$.
	It is well known~\cite{Carlet2021} that for a bent function $f \in \mathcal{B}_n$, the maximum possible dimension of an $\mathcal{M}$-subspace is $n/2$. Bent functions attaining this bound are precisely those belonging to the class $\mathcal{M}^\#$, as established by Lemma~\ref{lem M-M second}.  For any Maiorana-McFarland bent function of the form  
	\[
	f(x,y) = x \cdot \pi(y) + h(y),
	\]  
	where $(x,y) \in \F_2^m \times \F_2^m$, the vector space $\F_2^m \times \{0_m\}$ is an $\mathcal{M}$-subspace of $f$, as noted by Dillon~\cite{Dillon1974}. However, in general, this vector space — referred to as the \textit{canonical} $\mathcal{M}$-subspace — is not necessarily the unique $\mathcal{M}$-subspace of dimension $n/2$ for a Maiorana-McFarland bent function $f\in\mathcal{B}_n$.
	
	We will use the notation $f=f_1||f_2||f_3||f_4$ to denote the canonical concatenation of $f_i \in \cB_n$ for specifying the function $f \in \cB_{n+2}$. This precisely means that the restrictions of $f(x,x_{n+1}, x_{n+2})$, where $x \in \F_2^n$ and $x_{n+1}, x_{n+2} \in \F_2$, are given  as: $f(x,0,0)=f_1(x)$,  
	$f(x,0,1)=f_2(x)$, $f(x,1,0)=f_3(x)$, and $f(x,1,1)=f_4(x)$.

\subsection{Bent functions in  $\mathcal{GMM}_{\frac{n}{2}+k}$}\label{sec:conc}

For the purpose of this article, we define the (strict) generalized Maiorana-McFarland class ($\mathcal {GMM}_{\frac{n}{2}+k}$) for even $n=2m$ as below. 

\begin{defi}\label{def:GMMclass} 
 Let  $n=2m$ be an even positive integer and  $0\leq k\leq n/2-1.$ The set of all Boolean functions $f \colon \F_2^{\frac{n}{2}-k} \times \F_2^{\frac{n}{2}+k} \rightarrow \F_2$ of the form
\begin{eqnarray}\label{eq:Mextend}
f(x,y)=x\cdot\phi(y) +
h(y), \;\; x \in \F_2^{\frac{n}{2}-k}, y \in \F_2^{\frac{n}{2}+k},
\end{eqnarray}
is called the strict $\mathcal{GMM}_{\frac{n}{2}+k}$ class, where $\phi
 \colon \F_2^{\frac{n}{2}+k} \rightarrow \F_2^{\frac{n}{2}-k}$  and  $h\in
\cB_{\frac{n}{2}+k}$, for $k=0,1,\ldots,n/2-1.$

\end{defi}
\begin{rem}\label{rem:GMMcases}
When   $k=0$, the class defined by Eq. \eqref{eq:Mextend} corresponds to the $\mathcal{\cM}$
    class of bent functions if and only if $\phi$ permutes $\F_2^{\frac{n}{2}}$. In the extreme case, when $k=n/2-1$,
    the class defined by Eq. \eqref{eq:Mextend} 
    covers all Boolean functions on $\F_2^{n}$. Notice also that $f \colon \F_2^{\frac{n}{2}+k} \times \F_2^{\frac{n}{2}-k} \rightarrow \F_2$, for $k>0$, is not interesting to us since no bent function can be defined in this case. 

\end{rem}

In agreement with Remark~\ref{rem:GMMcases}, it was deduced \cite{Carlet2021,Polujan2020,DCC2022} that if a Boolean function $f \in \cB_n$ belongs to  $\mathcal{GMM}_{\frac{n}{2}+k}$,
then $f$ belongs to $\mathcal{GMM}_{\frac{n}{2}+k+1}$. Thereby,
if $k_1<k_2$ and $f$ belongs to  $\mathcal{GMM}_{\frac{n}{2}+k_1}$ then $f$ belongs to  $\mathcal{GMM}_{\frac{n}{2}+k_2}$. 
This result implies that the standard $\cM$ class of bent functions, i.e., $\mathcal{GMM}_{\frac{n}{2}}$, is embedded in $\mathcal{GMM}_{\frac{n}{2}+k}$ for any $k>0$.
A sufficient condition that $f \in \mathcal{GMM}_{\frac{n}{2}+k}$ is a bent function  was provided in~\cite{DCC2022}, which was also shown to be necessary in \cite[Theorem 6.40]{Logachev2012} (using somewhat different terminology).

\begin{prop}\label{prop:suffic_g_k}\cite{DCC2022}
Let $n$ be even and  $f\in \mathcal{GMM}_
{\frac{n}{2}+k}$ be defined by Eq. (\ref{eq:Mextend}). If $\phi  \colon \F_2^{\frac{n}{2}+k} \rightarrow
\F_2^{\frac{n}{2}-k}$ is $2^{2k}$-to-1 mapping and
for any $\alpha \in \mathbb{F}_2^{\frac{n}{2}-k}$, $\beta\in \mathbb{F}_2^{\frac{n}{2}+k}$,
\begin{equation}\label{eq:bentcond}
\Big | \sum_{y\in \mathbb{F}_2^{\frac{n}{2}+k}: \ {\phi}(y)=\alpha} (-1)^{h(y) + \beta \cdot y}\Big|=2^k,
\end{equation}
then $f$ is a bent function.
\end{prop}
 
We notice  that in contrast to the standard $\mathcal{M}$ class of bent functions where the function $h$  on $\F_2^{n/2}$ is arbitrary, the bent condition given by Eq.~\eqref{eq:bentcond} implies certain restrictions on $h\in \cB_{n/2+k}$.

 Similarly to the class $\cM$, the completed version of $\cGM_{n/2+k}$ is denoted by $\cGM_{n/2+k}^\#$  and is obtained by applying the same transformation as in Eq.~\eqref{eq:completed}.  The following characterization of  $\cGM_{n/2+k}^\#$  (adopted for even $n$), stated as Proposition~54 in~\cite{Carlet2021}, can be used to show the exclusion from $\cM^\#$.
    \begin{prop}\cite{Carlet2021}\label{prop: genclasschar}
        An $n$-variable Boolean function $f$, with $n$ even, belongs to the completed $\cGM_{n/2+k}^\#$ class
    if and only if there exists an $(n/2-k)$-dimensional vector space $E$ such that $D_aD_bf$ is the null
    function for every $a, b \in E$.
    \end{prop}
Notice that since $\cGM_{n/2+k_1}^\#$ is embedded in $\cGM_{n/2+k_2}^\#$ whenever $k_2>k_1$, we cannot use this criterion directly to exclude that 
a bent function $f \not \in  \cGM_{n/2}^\#$ without proving that $f \in \cGM_{n/2+k_1}^\# \setminus \cM^\#$, for some $k_1>0$.

\section{Boolean functions in the $\mathcal{GMM}_{\frac{n}{2}+1}$ class that are bent}\label{sec:bentness}
In this section, we will first give both necessary and sufficient conditions that a function $f(x,y)=x \cdot \phi(y) + h(y) \in \cGM_{n/2+1}$ is bent, see Theorem \ref{theo: GMM+1condition}. This is a refinement of the result in \cite{Logachev2012}, since we prove that the preimage sets of $\phi \colon \F_2^{n/2+1} \to \F_2^{n/2-1}$ are necessarily 2-dimensional affine subspaces. Moreover, we also describe the exact form of the dual bent function $f^*$ whenever $f$ is bent. For a fixed 4-to-1 mapping $\phi$ that partitions $\F_2^{n/2+1}$ into 2-dimensional affine subspaces, we will provide the exact number of different functions $h$ so that $f(x,y)=x \cdot \phi(y) + h(y)$ is bent. Moreover, we provide a solution to an open problem stated recently in \cite{ISIT-IEEE2024} and demonstrate that it is possible to specify an $(n+2)$-variable  bent function $f=f_1||f_2||f_3||f_4 \in \cM^\#$ even though the $n$-variable bent functions $f_i$ are outside $\cM^\#$. We also show that every 4-to-1 mapping  $\phi\colon\F_2^{m+1}\to\F_2^{m-1}$ satisfying the conditions of Theorem~\ref{theo: GMM+1condition} is extendable to a permutation of $\F_2^{m+1}$. 
\subsection{Full characterization of bent functions in the $\mathcal{GMM}_{\frac{n}{2}+1}$ class}\label{sec:GMM+1}
In this section, we analyze the properties of the $\mathcal{GMM}_{\frac{n}{2}+1}$ class, which is structurally closest to the standard $\cM$ class, as its elements can be viewed 
as concatenations of affine functions in $(n/2-1)=m-1$ variables. We provide both necessary and sufficient conditions under which the function $f(x,y)=x\cdot\phi(y) + h(y)$  in $\cGM_{m+1}$ is bent. Moreover, for a fixed choice of the mapping $\phi$, we determine the exact number of possible functions $h$ for which $f$ is bent, and we explicitly determine the dual bent function of $f$.

\begin{theo} \label{theo: GMM+1condition}
Let $n=2m$ and let $f$ be a function in the $\cGM_{m+1}$ class defined by 
\begin{eqnarray}\label{eq:GMM+1}
f(x,y)=x\cdot\phi(y) + 
h(y), \;\; x \in \F_2^{m-1}, y \in \F_2^{m+1},
\end{eqnarray}
where $\phi
 \colon \F_2^{m+1} \rightarrow \F_2^{m-1}$  and  $h \colon \F_2^{m+1} \rightarrow \F_2$. Then, $f$ is a bent function if and only if
\begin{itemize}
\item the collection $\left\lbrace \phi^{-1}(a) \mid a \in \F_2^{m-1} \right\rbrace$ is a partition of $\F_2^{m+1}$ into $2$-dimensional affine subspaces (where by $\phi^{-1}(a)$ we denote the set $\{ y \in \F_2^{m+1}  \mid \phi(y)=a \}$), and
\item for every $a\in\F_2^{m-1}$, the restriction of $h$ on the set $\phi^{-1}(a)$ has odd weight.
\end{itemize}
\end{theo}
\begin{proof}
According to~\cite[Theorem 6.40]{Logachev2012}, the mapping $\phi$ is $4$-to-$1$. For $a \in \F_2^{m-1}$ and $b \in \F_2^{m+1}$, we compute
\begin{align*}
W_f(a,b)&=
\sum_{(x,y) \in \F_2^{m-1}\times \F_2^{m+1}}
\left( -1 \right) ^{x\cdot\phi(y) +
h(y) + (x,y)\cdot (a,b)} \\
&=\sum_{y\in \F_2^{m+1}}\left( -1 \right) ^{
h(y) + y\cdot b} \sum_{x\in \F_2^{m-1}} \left( -1 \right) ^{x\cdot(a + \phi(y))} \\
&=2^{m-1} \sum_{y\in \phi^{-1}(a)}\left( -1 \right) ^{
h(y) + y\cdot b}.
\end{align*}
Hence, $f$ is bent if and only if
$
\sum_{y\in \phi^{-1}(a)}\left( -1 \right) ^{
h(y) + y\cdot b}=\pm 2,
$
for all $a \in \F_2^{m-1}$ and $b \in \F_2^{m+1}$. Let $\phi^{-1}(a)= \{ y_1,y_2,y_3,y_4 \}$. To avoid confusion, in the rest of the proof we use $\oplus$ to denote addition modulo two, whereas $\sum$ denotes summation over integers.
Then,
\begin{align*}
\sum_{y\in \phi^{-1}(a)}\left( -1 \right) ^{
h(y) \oplus y\cdot b} &= (-1)^{
h(y_1) \oplus y_1 \cdot b}+(-1)^{
h(y_2) \oplus y_2 \cdot b}+(-1)^{
h(y_3) \oplus y_4 \cdot b}+(-1)^{
h(y_4) \oplus y_4 \cdot b} \\
&= 4-2 \sum_{i=1}^4 (h(y_i) \oplus y_i \cdot b).
\end{align*}
Hence, $
\sum_{y\in \phi^{-1}(a)}\left( -1 \right) ^{
h(y) \oplus y\cdot b}=\pm 2,
$ if and only if $\sum_{i=1}^4 (h(y_i) \oplus y_i \cdot b)$ is odd,  
that is, the Boolean function $\bigoplus_{i=1}^4 (h(y_i) \oplus y_i \cdot b)$ is equal to $1$, for all $b \in \F_2^{m+1}$.  Setting $b=0_{m+1}$, we get that $\sum_{i=1}^4 h(y_i)$ is odd, i.e., the restriction of $h$ on the set $\phi^{-1}(a)$ has odd weight. Consequently, because $\bigoplus_{i=1}^4 h(y_i) \oplus \bigoplus_{i=1}^4 y_i \cdot b=1$, for all $b \in \F_2^{m+1}$, we deduce that $(\bigoplus_{i=1}^4 y_i) \cdot b=0$, for all $b \in \F_2^{m+1}$, so $\bigoplus_{i=1}^4 y_i=0_{m+1}$. A set of $4$ vectors $\{ y_1,y_2,y_3,y_4 \} \subset \F_2^{m+1}$ is an affine subspace of $\F_2^{m+1}$ if and only if $\bigoplus_{i=1}^4 y_i=0_{m+1}$, and this concludes the proof. 
\end{proof}
\begin{rem}
	It is an interesting problem to specify similar conditions regarding the decomposition of $\F_2^{m+k}$, for $k>1$, into $2k$-dimensional affine subspaces 
	as in Theorem  \ref{theo: GMM+1condition}. In particular, using the fact that $\F_2^{m+k} =\F_2^{k-1}\times  \F_2^{m+1}$ and that $\cGM_{m+1} \subset \cGM_{m+k}$, there is a possibility of identifying the class $\cGM_{m+1}$ within $\cGM_{m+k}$. This aspect is important when specifying bent functions in  $\cGM_{m+k}$ so that they do not intersect with those in $\cGM_{m+1}$ . 
\end{rem}
\begin{cor}\label{cor:weightofh}
Let $f$ be a bent function defined by Eq. \eqref{eq:GMM+1}. Then, the Hamming weight of $h$ satisfies $2^{m-1} \leq \operatorname{wt}(h) \leq 3 \cdot 2^{m-1}$.
\end{cor}
\begin{proof}
Denote by $N_3$ the set of vectors $a$ in $\F_2^{m-1}$ for which the restriction of $h$ on $\phi^{-1}(a)$ has weight $3$. Theorem~\ref{theo: GMM+1condition} implies that, for all vectors $b \in \F_2^{m-1} \setminus N_3$, the restriction of $h$ on $\phi^{-1}(b)$ has weight 1. Since $\left\lbrace \phi^{-1}(v) \mid v \in \F_2^{m-1} \right\rbrace$ is a partition of $\F_2^{m+1}$, the weight of $h$ is equal to (where ``$| \cdot |$'' denotes the cardinality) $$3 \cdot |N_3|+ 1 \cdot (2^{m-1}-|N_3|)= 2^{m-1}+2 \cdot |N_3|,$$
and since $0 \leq |N_3| \leq 2^{m-1}$, the result follows. 
    \end{proof}

\begin{theo}\label{theo: dualform}
Let $f$ be a bent function on $\F_2^{2m}$, thus satisfying the conditions in Theorem \ref{theo: GMM+1condition}, defined  by
\begin{eqnarray*}
f(x,y)=x\cdot\phi(y) + 
h(y), \;\; x \in \F_2^{m-1}, y \in \F_2^{m+1},
\end{eqnarray*}
where $\phi
\colon\F_2^{m+1} \rightarrow \F_2^{m-1}$  and  $h \colon \F_2^{m+1} \rightarrow \F_2$. Let $\prec$ denote any total ordering on $\F_2^{m+1}$ (in particular, we can take the lexicographic order on $\F_2^{m+1}$).  Then, the dual function $f^*$ of $f$ is given by
\begin{equation} \label{eq:dualform+1}
    f^*(a,b)= \sum_{\substack{y,y' \in \phi^{-1}(a) \\ y \precneqq y'}}(h(y)+y \cdot b)(h(y')+y' \cdot b),
\end{equation}
for all  $a\in \F_2^{m-1}, b \in \F_2^{m+1}$.
\end{theo}
\begin{proof}
From the proof of Theorem \ref{theo: GMM+1condition}, we have
\begin{equation*}
    W_f(a,b)=2^{m-1} \sum_{y\in \phi^{-1}(a)}\left( -1 \right) ^{
h(y) \oplus y\cdot b},
\end{equation*}
and that, for $\phi^{-1}(a)= \{y_1,y_2,y_3,y_4\}$, the vector $$(h(y_1) + y_1\cdot b, \; h(y_2) + y_2\cdot b, \; h(y_3) + y_3\cdot b, \; h(y_4) + y_4\cdot b) \in \F_2^4$$
has odd weight, for all  $a\in \F_2^{m-1}, b \in \F_2^{m+1}$.
Let $g \colon \F_2^4 \to \F_2$ be the function defined by $$g(z_1,z_2,z_3,z_4)= z_1z_2+z_1z_3+z_1z_4+z_2z_3+z_2z_4+z_3z_4.$$
The function $g$ is such that $g(z)=1$ for vectors $z\in\F_2^4$ with $\operatorname{wt}(z)=3$, and $g(z)=0$ for vectors $z\in\F_2^4$ with $\operatorname{wt}(z)=1$.
Then, for vectors $(z_1,z_2,z_3,z_4) \in \F_2^4$ of odd weight, we have
\begin{equation*}
    (-1)^{z_1}+(-1)^{z_2}+(-1)^{z_3}+(-1)^{z_4}= 2 (-1)^{g(z_1,z_2,z_3,z_4)},
\end{equation*}
and consequently
\begin{equation*}
    W_f(a,b)=2^{m} \left( -1 \right)^{g(h(y_1) + y_1\cdot b, \; h(y_2) + y_2\cdot b, \; h(y_3) + y_3\cdot b, \; h(y_4) + y_4\cdot b)},
\end{equation*}
and Eq. \eqref{eq:dualform+1} follows. 
\end{proof}

\begin{rem}
    Note that any function $g' \colon \F_2^4 \to \F_2$ such that $g'(z)=1$ for vectors $z\in\F_2^4$ with $\operatorname{wt}(z)=3$, and $g'(z)=0$ for vectors $z\in\F_2^4$ with $\operatorname{wt}(z)=1$, could have been used in the proof of Theorem \ref{theo: dualform} to obtain a version of Eq.~\eqref{eq:dualform+1} and to write $f^*$ in a slightly different form. There are exactly $2^8=256$ such functions, since the values on inputs of even weight can be chosen arbitrarily. Furthermore, from the proof of Theorem \ref{theo: dualform} we know that the weight of $$(h(y_1) + y_1\cdot b, \; h(y_2) + y_2\cdot b, \; h(y_3) + y_3\cdot b, \; h(y_4) + y_4\cdot b) \in \F_2^4$$ is odd, for all $b \in \F_2^{m+1}$ and $a \in \F_2^{m-1}$, where $\phi^{-1}(a)= \{y_1,y_2,y_3,y_4\}$. Since all the functions $g' \colon \F_2^4 \to \F_2$, used to represent $f^*$, agree on the vectors of odd weight, although we have different representations of $f^*$, the values of $f^*$ remain unchanged and unique as expected. We have chosen $g$ in Theorem \ref{theo: dualform} for which the obtained form of the dual $f^*$ seems to be the simplest one.

\end{rem}

    \begin{cor}\label{cor: thenumberofh}
        Let $\phi \colon \F_2^{m+1} \to \F_2^{m-1}$ be a 4-to-1 mapping such that $\left\lbrace \phi^{-1}(a) \mid a \in \F_2^{m-1} \right\rbrace$ is a partition of $\F_2^{m+1}$ into $2$-dimensional affine subspaces. Then, there are exactly $2^{3 \cdot 2^{m-1}}$ functions $h \colon \F_2^{m+1} \to \F_2$ such that the function $f$ defined by
        \begin{eqnarray} \label{eq: counting}
        f(x,y)=x\cdot\phi(y) +
        h(y), \;\; x \in \F_2^{m-1}, y \in \F_2^{m+1},
        \end{eqnarray}
        is bent.
    \end{cor}
    \begin{proof}
        From Theorem \ref{theo: GMM+1condition}, once the mapping $\phi$ is fixed, for every $a\in\F_2^{m-1}$, we are free to choose whether the restriction of $h$ to $\phi^{-1}(a)$ will have weight $1$ or $3$.  Then, if the weight of the restriction is $1$ we have $4$ options to choose which element of  $\phi^{-1}(a)$  maps to $1$, alternatively,  if the weight of the restriction is $3$ we have $4$ options to choose which element of  $\phi^{-1}(a)$  maps to $0$. 
        Hence, for every $a\in\F_2^{m-1}$, we have $8$ options for the restriction of $h$ to $\phi^{-1}(a)$. We deduce that there are exactly $8^{2^{m-1}}=2^{3 \cdot 2^{m-1}}$ options for the function $h$ so that the function $f$ defined by Eq.~\eqref{eq: counting} is bent. 
    \end{proof}
 
\begin{ex}\label{ex:different_h}
	Let $m = 4$, and consider the mapping $\phi \colon \F_2^5 \to \F_2^3$ whose coordinate functions are given by (as rows):
	\begin{equation*}
		\begin{split}
			\phi(y_1,y_2,y_3,y_4,y_5)&=\begin{pmatrix}
				y_2 + y_3 + y_1 y_3 + y_1 y_4 + y_1 y_2 y_4 + y_3 y_4 + y_1 y_3 y_4 + y_2 y_3 y_4 + 
				y_1 y_2 y_5 + y_3 y_5\\ 
				y_1 + y_1 y_3 + y_2 y_3 + y_1 y_4 + y_2 y_4 + y_3 y_4 + y_2 y_3 y_4 + y_2 y_5 + 
				y_1 y_2 y_5 + y_1 y_3 y_5\\ 
				y_1 y_2 + y_3 + y_1 y_4 + y_2 y_4 + y_1 y_2 y_4 + y_2 y_3 y_4 + y_1 y_5 + 
				y_1 y_3 y_5 + y_2 y_3 y_5
			\end{pmatrix}^T.
		\end{split}
	\end{equation*}
    It can be verified that  $f(x,y)=x\cdot \phi(y)$, where $x\in\F_2^3$ and $y\in\F_2^5$, is a bent function with a unique $\mathcal{M}$-subspace of dimension $3$ (namely, the canonical $\mathcal{M}$-subspace $\F_2^3 \times \{0_5 \}$). Consequently, $f \in \mathcal{GMM}_{5} \setminus \cM^\#$. Define now $f_h(x,y)=x\cdot \phi(y)+h(y)$, where $h\in\mathcal{B}_5$ is affine-free and not of maximal degree 5 (due to the upper bound on algebraic degree of bent functions). We note that the number of such functions $h$ in $\mathcal{B}_5$ is equal to $2^{25}$. 
	Using a computer algebra system, we could confirm that out of $2^{25}$ such functions there are  $2^{18}$ bent functions, which  is exactly the value $2^{3 \cdot 2^{m-1}}=2^{24}$ in Corollary \ref{cor: thenumberofh} (modulo the number of affine terms $2^6$). 
	Moreover, the decomposition of $\F_2^5$ w.r.t. 
	$\phi^{-1}(a)$ when $a \in \F_2^3$ is given by: 
	\begin{eqnarray*}
		& & A_0 =\{00000, 00001, 00010, 00011\}, \quad A_1=\{00101, 01110, 10100, 11111\},\\
		& & A_2=\{10000, 10111, 11001, 11110\}, \quad 	A_3=\{00110, 01100, 10001, 11011\}, \\
		& & A_4=\{01000, 01101, 10011, 10110\}, \quad  A_5=\{00100, 01011, 10010, 11101\}, \\
		&& 	A_6=\{01001, 01111, 11010, 11100\}, \quad	A_7=\{00111, 01010, 10101, 11000\}.
	\end{eqnarray*}
	
\end{ex}
	\begin{rem}\label{rem:nodependency-on-h}
		For the mapping $\phi\colon\F_2^5\to\F_2^3$ in Example \ref{ex:different_h}, we constructed a random sample of $2^{10}$ bent functions $f_h$ of the form $f_h(x,y)=x\cdot \phi(y)+h(y)$, all of which are outside $\mathcal{M}^\#$. Essentially, the function $h$ only seems to affect the bentness of $f_h$ and appears to be irrelevant with respect to the class inclusion in $\mathcal{M}^\#$.
	\end{rem}
\begin{rem}\label{rem:structure-of-f}
	Notably, among  $2^{25}$ functions $f_h$ considered in Example \ref{ex:different_h}  which  are not bent (with fixed $\phi$ and varying $h$), there are exactly $2^{18}$ semi-bent functions and $2^{19}\cdot 3^2\cdot 7$ 5-valued spectra functions. Therefore, there are no other possibilities for $f_h$ other than bent, semi-bent, or 5-valued spectra functions. This essentially follows from Theorem~\ref{theo: GMM+1condition} and the computation of the Walsh spectrum which is given by $W_f(a,b)=2^{m-1} \sum_{y\in \phi^{-1}(a)}\left( -1 \right) ^{
		h(y) + y\cdot b}$, where $|\phi^{-1}(a)|=4$, for any $a \in \F_2^{m-1}$. Hence, similar conditions related to semi-bent and 5-valued spectra functions in the $\mathcal{GMM}_{m-1}$ class can be specified through the requirements imposed on $h$.
\end{rem}

The fact that all tested bent functions in Example~\ref{ex:different_h} are outside $\cM^\#$ indicates that, due to the choice of $\phi$, the partition of $\F_2^{m+1}$ into 2-dimensional affine subspaces is non-trivial (more precisely, it is a proper partition of $\F_2^{m+1}$, see Definition~\ref{def: properdec}), which is addressed later in Section \ref{subsec:bentinMM}. 

\subsection{Concatenating functions in $\mathcal{GMM}_{\frac{n}{2}+1}$ with the same partition}
In this section, we will demonstrate that the case of defining $f=f_1||f_2||f_3||f_4$, where all bent functions $f_i(x,y)=x \cdot \phi(y) + h_i(y) \in \cGM_{\frac{n}{2}+1}$ defined on $\F_2^n$ use the same partition (thus the same mapping $\phi$) leads to a very important result. More precisely, we will show that every bent function in the $\cGM_{m+1}$ class (including those outside $\cM^\#$) is a part of a bent function in the $\cM^\#$ class in a larger number of variables. Alternatively,  it is possible to design (many) bent functions $f$ on $\F_2^{n+2}$ in the strict Maiorana-McFarland class, whose restrictions to $\F_2^n \times \{ (0,0) \}$ are bent functions outside the $\mathcal{M}^\#$ class. This result actually positively answers the open problem stated in \cite{ISIT-IEEE2024}: ``Does there exist a bent function
$f = f_1 || f_2 || f_3 || f_4$ on $\mathbb{F}_2^{n+2}$
in the $\mathcal{M}^{\#}$ class, such that each $f_i \in \mathcal{B}_n$ is bent and satisfies $f_i \notin \mathcal{M}^{\#}$?''
\begin{theo}\label{theo: dualofconcatenation}
    Let $\phi  \colon \F_2^{m+1} \to \F_2^{m-1}$ be a mapping such that $\{ \phi^{-1}(a)\}_{a \in \F_2^{m-1}}$ is a partition of $\F_2^{m+1}$ into $2$-dimensional affine subspaces. Let $h_i \colon \F_2^m \to \F_2$, $i=1,\ldots,4$, be such that
    $$f_i(x,y)= x \cdot \phi(y) + h_i(y), \; x \in \F_2^{m-1}, y \in \F_2^{m+1},$$
    is bent. Then, the function $f=f_1||f_2||f_3||f_4 \in \cB_{2m+2}$ is bent if and only if for every $a \in \F_2^{m-1}$ the following two conditions hold: 
    \begin{enumerate}[I)]
        \item $\displaystyle\sum_{i=1}^4 \sum_{\substack{y,y' \in \phi^{-1}(a) \\ y \precneqq y'}}h_i(y)h_i(y')=1, $
        \item $\displaystyle \sum_{i=1}^4h_i(y)=0, \text{ for all } y \in \phi^{-1}(a)$; or $\;\displaystyle\sum_{i=1}^4h_i(y)=1, \text{ for all } y \in \phi^{-1}(a)$. 
    \end{enumerate}    
\end{theo}

\begin{proof}
Since the functions $f_i$, $i=1, \dots ,4$, are bent, the function $f$ is bent if and only if $f_1^*+f_2^*+f_3^*+f_4^*=1$, see  \cite{SHCF}. From Theorem \ref{theo: dualform}, we have that the duals of $f_i$ have the following form:
 \begin{align*}   f_i^*(a,b) &= \sum_{\substack{y,y' \in \phi^{-1}(a) \\ y \precneqq y'}}(h_i(y)+y \cdot b)(h_i(y')+y' \cdot b) \\
 &= \sum_{\substack{y,y' \in \phi^{-1}(a) \\ y \precneqq y'}} \left( h_i(y)h_i(y') + (h_i(y)y') \cdot b + (h_i(y')y) \cdot b + (y \cdot b)(y' \cdot b) \right) \\
 &=\sum_{\substack{y,y' \in \phi^{-1}(a) \\ y \precneqq y'}} \left( h_i(y)h_i(y') +  \left( h_i(y)y' + h_i(y')y \right) \cdot b + (y \cdot b)(y' \cdot b) \right),
 \end{align*}
for every $a\in \F_2^{m-1}$ and $ b \in \F_2^{m+1}$.
Let $\phi^{-1}(a)=\{y_1,y_2,y_3,y_4 \}$. Since all the functions $f_i$ are bent, from Theorem \ref{theo: GMM+1condition} we have $h_i(y_1)+h_i(y_2)+h_i(y_3)+h_i(y_4)=1$. Hence, from
\begin{align*}
 \sum_{\substack{y,y' \in \phi^{-1}(a) \\ y \precneqq y'}} \left( h_i(y)y' + h_i(y')y \right) &= \left( h_i(y_2)+h_i(y_3)+h_i(y_4) \right) y_1 + \left( h_i(y_1)+h_i(y_3)+h_i(y_4) \right) y_2 +  \\
&+ \left( h_i(y_1)+h_i(y_2)+h_i(y_4) \right) y_3 + \left( h_i(y_1)+h_i(y_2)+h_i(y_3) \right) y_4,
\end{align*}
it follows that 
\begin{align*}
 \sum_{\substack{y,y' \in \phi^{-1}(a) \\ y \precneqq y'}} \left( h_i(y)y' + h_i(y')y \right) &= \left( h_i(y_1)+1 \right) y_1 + \left( h_i(y_2)+1 \right) y_2 +  \left( h_i(y_3) +1 \right) y_3 + \left( h_i(y_4) +1 \right) y_4 \\
 &= \sum_{y \in \phi^{-1}(a)}h_i(y)y + \sum_{y \in \phi^{-1}(a)}y = \sum_{y \in \phi^{-1}(a)}h_i(y)y,
\end{align*}
where the bentness of  $f_i$  implies that $\phi^{-1}(a)$ is a $2$-dimensional affine subspace by Theorem~\ref{theo: GMM+1condition}, hence $\sum_{y \in \phi^{-1}(a)}y= 0_{m+1}$. On the other hand, the term $(y \cdot b)(y' \cdot b)$ does not depend on $i$, hence for the sum of the duals we obtain
\begin{align*}  \left( \sum_{i=1}^4 f_i^* \right)(a,b) &= \sum_{i=1}^4 \sum_{\substack{y,y' \in \phi^{-1}(a) \\ y \precneqq y'}}h_i(y)h_i(y') + \left(\sum_{i=1}^4 \sum_{y \in \phi^{-1}(a)}h_i(y)y \right) \cdot b.
 \end{align*}
Consequently, since $a\in \F_2^{m-1}$ and $ b \in \F_2^{m+1}$ are arbitrary, $\sum_{i=1}^4 f_i^*=1$ if and only if 
$$
\sum_{i=1}^4 \sum_{\substack{y,y' \in \phi^{-1}(a) \\ y \precneqq y'}} h_i(y)h_i(y') =1 \; \text{ and } 
\sum_{i=1}^4 \sum_{y \in \phi^{-1}(a)} h_i(y)y = 0_{m+1}.
$$
If $\sum_{i=1}^4h_i(y)=0, \text{ for all } y \in \phi^{-1}(a)$, then $\sum_{y \in \phi^{-1}(a)} \sum_{i=1}^4 h_i(y)y=0_{m+1}$, and if $\sum_{i=1}^4h_i(y)=1, \text{ for all } y \in \phi^{-1}(a)$, then $\sum_{y \in \phi^{-1}(a)} \sum_{i=1}^4 h_i(y)y= \sum_{y \in \phi^{-1}(a)} y= 0_{m+1}$ (because $\phi^{-1}(a)$ is a $2$-dimensional affine subspace). 

On the other hand, if $\sum_{i=1}^4h_i(y)=1$ for an odd number of elements $y \in \phi^{-1}(a)$, then $ \sum_{y \in \phi^{-1}(a)} \sum_{i=1}^4h_i(y)=1$, however, since $f_i$, $i=1,2,3,4$, are bent, from Theorem \ref{theo: GMM+1condition} we have that $ \sum_{y \in \phi^{-1}(a)}h_i(y)=1$, and hence $ \sum_{y \in \phi^{-1}(a)} \sum_{i=1}^4h_i(y)=\sum_{i=1}^41=0$, a contradiction. If $\sum_{i=1}^4h_i(y)=1$ for exactly two different elements  $y,y' \in \phi^{-1}(a)$, then 
$\sum_{y \in \phi^{-1}(a)} \sum_{i=1}^4 h_i(y)y = 0_{m+1}$ would imply that $y=y'$, a contradiction. 
We conclude that 
$\sum_{i=1}^4 \sum_{y \in \phi^{-1}(a)} h_i(y)y = 0_{m+1}$ if and only if $\sum_{i=1}^4h_i(y)=0$, for all  $y \in \phi^{-1}(a)$, or $\sum_{i=1}^4h_i(y)=1$, for all  $y \in \phi^{-1}(a)$, and this concludes the proof.
\end{proof}

\begin{rem}\label{rem: matrixHa}
Using the same notation as in Theorem \ref{theo: dualofconcatenation} and defining the function $g$ as in the proof of Theorem \ref{theo: dualform}, i.e., let $g \colon \F_2^4 \to \F_2$ be the function defined by $$g(z_1,z_2,z_3,z_4)= z_1z_2+z_1z_3+z_1z_4+z_2z_3+z_2z_4+z_3z_4, $$
then, for $\phi^{-1}(a)=\{ y_1, y_2, y_3, y_4 \} \subset \F_2^{m+1}$, we can state the first condition in Theorem \ref{theo: dualofconcatenation} as 
$$
\sum_{i=1}^4 g \left( h_i(y_1),h_i(y_2),h_i(y_3),h_i(y_4) \right)=1.
$$
Furthermore, for a $4$-to-$1$ mapping $\phi\colon \F_2^{m+1} \to \F_2^{m-1}$ and every $a \in \F_2^{m-1}$, we can define a matrix $H_a$ in the following way
$$H_a= \begin{pmatrix}
    h_1(y_1) & h_1(y_2) & h_1(y_3) & h_1(y_4) \\
    h_2(y_1) & h_2(y_2) & h_2(y_3) & h_2(y_4) \\
    h_3(y_1) & h_3(y_2) & h_3(y_3) & h_3(y_4) \\
    h_4(y_1) & h_4(y_2) & h_4(y_3) & h_4(y_4) 
\end{pmatrix},$$
where $ \phi^{-1}(a)= \{ y_1, y_2, y_3, y_4 \}\subset \F_2^{m+1}$.
Then, assuming that $\phi^{-1}(a)$ is a $2$-dimensional affine subspace for every $a \in \F_2^{m-1}$,
the functions $f(x,y)=x \cdot \phi(y) + h_i(y)$, $i=1,\ldots,4$, are bent if and only if the sum of the columns of the matrix $H_a$ is equal to $\begin{pmatrix}
     1 \\ 1 \\ 1 \\ 1
\end{pmatrix}$, for all $a \in \F_2^{m-1}$. Similarly, using the matrix $H_a$, we can state the two conditions in Theorem \ref{theo: dualofconcatenation}. Namely, the second condition in Theorem \ref{theo: dualofconcatenation} is satisfied if and only if the sum of the rows of the matrix $H_a$ is equal to $(0,0,0,0)$
or $(1,1,1,1)$, for every $a \in \F_2^{m-1}$. To state the first condition in Theorem  \ref{theo: dualofconcatenation} in terms of $H_a$, notice that the function $g(z_1,z_2,z_3,z_4)= z_1z_2+z_1z_3+z_1z_4+z_2z_3+z_2z_4+z_3z_4 $ is equal to $1$ if the weight of $(z_1,z_2,z_3,z_4)$ is $3$ and to $0$ if the weight of $(z_1,z_2,z_3,z_4)$ is $1$. Hence, the first condition in Theorem \ref{theo: dualofconcatenation} (which is equivalent to $\sum_{i=1}^4 g \left( h_i(y_1),h_i(y_2),h_i(y_3),h_i(y_4) \right)=1$) 
is satisfied if and only if there is an odd number of rows in $H_a$ with weight equal to $3$ (i.e., with 3 ones), that is either we have 3 rows with 3 ones and one row with 1 one, or 1 row with 3 ones and 3 rows with 1 one, for every $a \in \F_2^{m-1}$.
\end{rem}
We will need the following result form \cite{ISIT-IEEE2024,IEEE2024}.
\begin{cor}\label{cor: 4concatenationoutsideMM}\cite{ISIT-IEEE2024,IEEE2024}
	Let $f=f_1 \vert \vert f_2 \vert \vert f_3 \vert \vert f_4\in\mathcal{B}_{n+2}$ be the concatenation of $f_1, \ldots ,f_4 \in \mathcal{B}_n$ and assume that $f$ is bent; thus $f_i$ are bent, semi-bent or five-valued spectra functions. Then, $f$ is outside of the $\cM^{\#}$ class if and only if the following conditions hold:
	\begin{enumerate}[a)]
		\item The functions $f_1, \dots ,f_4$ do not share a common $(n/2+1)$-dimensional $\cM$-subspace; 
		\item There are no common $(n/2)$-dimensional $\cM$-subspaces $V \subset \F_2^n$ of $f_1, \dots ,f_4$ such that there is an element $a\in \F_2^n$ for which
		\begin{equation}
			\begin{gathered}
				D_vf_1+D_vf_2^a= D_vf_3+D_vf_4^a=0, \text{ for all } v \in V, \text{ or} \\
				D_vf_1+D_vf_3^a= D_vf_2+D_vf_4^a=0, \text{ for all } v \in V, \text{ or} \\
				D_vf_1+D_vf_4^a= D_vf_2+D_vf_3^a=0, \text{ for all } v \in V.
			\end{gathered}
		\end{equation}
		\item There are no common $(n/2-1)$-dimensional $\cM$-subspaces $V \subset \F_2^n$ of $f_1, \dots ,f_4$ such that there are elements $a,b\in \F_2^n$ (not necessarily different), for which
		\begin{equation}
			\begin{split}
				&D_vf_1+D_vf_3^a= D_vf_2+D_vf_4^a=D_vf_1+D_vf_2^b =D_vf_3+D_vf_4^b=0, \text{ for all } v \in V, \text{ and} \\
				&f_1(x)+f_2(x+b)+f_3(x+a)+f_4(x+a+b)=0, \text{ for all } x \in \F_2^n.
			\end{split}
		\end{equation}
	\end{enumerate}
\end{cor}
The main result of this section can now be stated as follows. 
\begin{theo}\label{th:outsidein}
	Let $f_1\in\mathcal{B}_{2m}$ be a bent function in the $\cGM_{m+1}$ class. Then, there exist bent functions $f_2,f_3,f_4\in\mathcal{B}_{2m}$ such that  $f=f_1 \vert \vert f_2 \vert \vert f_3 \vert \vert f_4 \in \cB_{2m+2}$ is a bent function in the $\cM^{\#}$ class. Consequently, every function in the $\cGM_{m+1}$ class (even those outside $\cM^{\#}$) is a  part of a function in the $\cM^{\#}$ class in a larger number of variables.
\end{theo}
\begin{proof}
Let $f_1(x,y)=x \cdot \phi(y)+h_1(y)$, where $x \in \F_2^{m-1}$ and $y \in \F_2^{m+1}$, be a bent function in the $\cGM_{m+1}$ class.
We will define bent functions $f_i(x,y)=x \cdot \phi(y)+h_i(y)$, where $i=2,3,4$, for some appropriately chosen $h_i$, thus of the same form as $f_1$, so that $f=f_1 \vert \vert f_2 \vert \vert f_3 \vert \vert f_4 \in \cB_{2m+2}$ is a bent function in the $\cM^{\#}$ class.

Set $a=0_{2m}$ and $b=0_{2m}$ in Corollary \ref{cor: 4concatenationoutsideMM}-(c) and $V=\F_2^{m-1} \times \{ 0_{m+1} \}$. Then, because $D_vf_i= v_x \cdot \phi (y)$ for $v =(v_x, 0_{m+1}) \in V$, the first part of c) in Corollary \ref{cor: 4concatenationoutsideMM} is satisfied, i.e., $D_vf_i+D_vf_j=0$, for all $v \in V$. The second part of c) in Corollary \ref{cor: 4concatenationoutsideMM} becomes $\sum_{i=1}^4f_i(x,y)= \sum_{i=1}^4h_i(y)$. Hence, from Corollary \ref{cor: 4concatenationoutsideMM} it follows that in order to get a function $f$ in $\cM^{\#}$, we need to find $h_2,h_3$ and $h_4$ such that $\sum_{i=1}^4h_i=0$. Of course, we still need to satisfy the conditions of Theorem \ref{theo: dualofconcatenation}. In order to achieve that we use the observations from Remark \ref{rem: matrixHa}, namely, (using the same notation as in Remark \ref{rem: matrixHa}), we will construct the matrix $H_a$ with the desired properties based on the values of $h_1$ on $\phi^{-1}(a)$, where $a \in \F_2^{m-1}$ (notice that we used  $a=0_{2m}$ before to make a straightforward connection to Corollary 
\ref{cor: 4concatenationoutsideMM}). For $a\in \F_2^{m-1}$, let $\phi^{-1}(a)=\{y_1,y_2,y_3,y_4 \}$. Assume first that the weight of $h_1$ on $\phi^{-1}(a)$ is $3$, w.l.o.g., (because the order is not important) so that $$h_1(y_1)=h_1(y_2)=h_1(y_3)=1; \quad\mbox{and}\quad h_1(y_4)=0.$$ 
Then,  we define $h_2(y_1)=1$ and $h_2(y_2)=h_2(y_3)=h_2(y_4)=0$; $h_3(y_2)=1$ and $h_3(y_1)=h_3(y_3)=h_3(y_4)=0$; and $h_4(y_3)=1$ and $h_4(y_1)=h_4(y_2)=h_4(y_4)=0$. Thus, the matrix $H_a$ becomes 
\begin{equation}\label{eq: Ha weight 3}
	H_a =
	\begin{pmatrix}
		1 & 1 & 1 & 0 \\
		1 & 0 & 0 & 0 \\
		0 & 1 & 0 & 0 \\
		0 & 0 & 1 & 0
	\end{pmatrix},
\end{equation}
and due to the observations in Remark \ref{rem: matrixHa} the conditions of Theorem \ref{theo: dualofconcatenation} are satisfied in this case because the sum of columns is $\begin{pmatrix}
     1 \\ 1 \\ 1 \\ 1
\end{pmatrix}$, the sum of rows is $(0,0,0,0)$ and we have an odd number of rows with 3 ones (one in this case). Notice that it is possible to define $H_a$ in different ways to obtain the same result.

Assume now that the weight of $h_1$ on $\phi^{-1}(a)$ is $1$ and w.l.o.g., $h_1(y_1)=1$ and $h_1(y_2)=h_1(y_3)=h_1(y_4)=0$. Then, we define $h_2(y_2)=1$ and $h_2(y_1)=h_2(y_3)=h_2(y_4)=0$; $h_3(y_3)=1$ and $h_3(y_1)=h_3(y_2)=h_3(y_4)=0$; $h_4(y_1)=h_4(y_2)=h_4(y_3)=1$ and $h_4(y_4)=0$. The matrix $H_a$ becomes 
\begin{equation}\label{eq: Ha weight 1}
	H_a =
	\begin{pmatrix}
		1 & 0 & 0 & 0 \\
		0 & 1 & 0 & 0 \\
		0 & 0 & 1 & 0 \\
		1 & 1 & 1 & 0
	\end{pmatrix},
\end{equation}
and as in the first case, from the observations in Remark \ref{rem: matrixHa} the conditions of Theorem \ref{theo: dualofconcatenation} are satisfied in this case as well.  We notice that the sum of columns is $\begin{pmatrix}
     1 \\ 1 \\ 1 \\ 1
\end{pmatrix}$, the sum of rows is $(0,0,0,0)$ and we have an odd number of rows with 3 ones (one in this case). Consequently, from Theorem \ref{theo: dualofconcatenation} it follows that $f=f_1 \vert \vert f_2 \vert \vert f_3 \vert \vert f_4$ is a bent function.
Since the sum of the rows of $H_a$ is always $(0,0,0,0)$, we deduce that $\sum_{i=1}^4f_i(x,y)= \sum_{i=1}^4h_i(y)=0$ for all $x \in \F_2^{m-1}$ and $y \in \F_2^{m+1}$, and so from Corollary \ref{cor: 4concatenationoutsideMM}-(c) it follows that $f$ is a function in the $\cM^{\#}$ class.
\end{proof}

	\begin{cor}
		For every even $n\ge 8$, there exist bent functions on $\F_2^{n+2}$ that belong to the $\mathcal{M}^\#$ class, whose restrictions to $\F_2^n \times \{ (0,0) \}$ are bent functions outside the $\mathcal{M}^\#$ class.
	\end{cor}
	\begin{proof}
        Theorem~\ref{th:outsidein} implies that, in order to establish the existence of  bent functions on $\F_2^{n+2}$ in $\mathcal{M}^\#$ whose restrictions to $\F_2^n \times \{ (0,0) \}$ are bent functions outside $\mathcal{M}^\#$, it is enough to establish the existence of bent functions on $\F_2^n$ in $\mathcal{GMM}_{\frac{n}{2}+1} \setminus \cM^\#$, for all even $n \geq 8$.
        In Theorem~\ref{th:4concatoutsideMM} the existence of bent functions on $\F_2^n$ in $\mathcal{GMM}_{\frac{n}{2}+1} \setminus \cM^\#$, for all even $n \geq 8$, is established, based on the existence of permutations on $\F_2^m$ with the property~\eqref{eq: P1}, for all integers $m \geq 3$. The existence of such permutations is established in \cite[Corollary 23]{PPKZ2023}.
	\end{proof}

\begin{ex}
	Let $\phi\colon\F_2^5\to\F_2^3$ be a 4-to-1 mapping satisfying the conditions of Theorem~\ref{theo: GMM+1condition}, which is defined by $\phi(y)=(\phi_1(y),\phi_2(y),\phi_3(y))$, where
	\begin{equation}\label{eq: the first phi}
		\begin{split}
			\phi_1(y)=&y_1 + y_1 y_2 y_3 + y_1 y_4 + y_2 y_4 + y_1 y_2 y_4 + y_3 y_4 + y_2 y_3 y_4 \\
			+& y_1 y_5 +
			y_2 y_5 + y_3 y_5 + y_2 y_3 y_5 + y_1 y_4 y_5 + y_2 y_4 y_5,\\
			\phi_2(y)=&y_3 + y_1 y_3 + y_1 y_2 y_3 + y_2 y_4 + y_2 y_3 y_4 + y_1 y_5 + y_2 y_5 + y_1 y_3 y_5 +
			y_3 y_4 y_5,\\
			\phi_3(y)=&y_2 + y_1 y_3 + y_2 y_3 + y_1 y_4 + y_1 y_2 y_4 + y_3 y_4 + y_1 y_3 y_4 + y_2 y_3 y_4\\
			+& y_1 y_5 + y_1 y_2 y_5 + y_3 y_5 + y_2 y_3 y_5 + y_2 y_4 y_5 + y_3 y_4 y_5.
		\end{split}
	\end{equation}
	Define four Boolean functions $h_i\in\mathcal{B}_5$ as in Theorem~\ref{th:outsidein} using the matrix defined by Eq.~\eqref{eq: Ha weight 1}:
		\begin{equation}
			\begin{split}
				h_1(y)=&
				1 + y_1 + y_1 y_3 + y_2 y_3 + y_4 + y_1 y_4 + y_1 y_3 y_4 + y_1 y_2 y_3 y_4 + y_5 + 
				y_1 y_2 y_5 \\
				+& y_3 y_5 + y_1 y_3 y_5 + y_4 y_5 + y_2 y_4 y_5 + y_3 y_4 y_5 + 
				y_1 y_3 y_4 y_5,\\
				h_2(y)=&
				y_1 y_2 + y_2 y_3 + y_1 y_4 + y_3 y_4 + y_5 + y_1 y_5 + y_2 y_5 + y_3 y_5 + 
				y_1 y_3 y_5 + y_2 y_3 y_5 \\
				+& y_4 y_5 + y_2 y_4 y_5 + y_1 y_2 y_4 y_5 + y_2 y_3 y_4 y_5, \\
				h_3(y)=&
				y_1 y_2 y_3 + y_1 y_3 y_4 + y_1 y_2 y_3 y_4 + y_2 y_5 + y_2 y_3 y_5 + y_4 y_5 + 
				y_3 y_4 y_5 + y_1 y_3 y_4 y_5, \\
				h_4(y)=&
				1 + y_1 + y_1 y_2 + y_1 y_3 + y_1 y_2 y_3 + y_4 + y_3 y_4 + y_1 y_5 + y_1 y_2 y_5 + 
				y_4 y_5 \\ 
				+& y_1 y_2 y_4 y_5 + y_2 y_3 y_4 y_5.
			\end{split}
	\end{equation}
    For each $i = 1, 2, 3$, the weight $\operatorname{wt}(h_i)$ is equal to 1 on any 2-dimensional affine subspace $\phi^{-1}(a)$. On the other hand, $\operatorname{wt}(h_4) = 3$ on the same subspaces. Using the above described mapping $\phi$ and the functions $h_i$, define four almost Maiorana-McFarland bent functions outside $\mathcal{M}^\#$ on $\F_2^8$ by $f_i(x,y)=x \cdot \phi(y)+h_i(y)$ (as in Theorem~\ref{th:outsidein}). Then, the resulting bent function $f=f_1 \vert \vert f_2 \vert \vert f_3 \vert \vert f_4$ on $\F_2^{10}$ is in the $\mathcal{M}^\#$ class. Notice that adding the constant all-one function to $f$, the case when $\operatorname{wt}(h_i)=3$, for $i=1,2,3,$ on the affine subspaces $\phi^{-1}(a)$ is also covered.
\end{ex}

\section{Almost Maiorana-McFarland bent functions and their inclusion in $\mathcal{M}^\#$}\label{sec:inclusion}
Even though Theorem \ref{theo: GMM+1condition} specifies the bent condition which is both necessary and sufficient, the main task is to distinguish these bent functions with respect to their inclusion in $\cM^\#$. In this context, we will provide a  criterion that must not be fulfilled if  a bent function $f \in \cGM_{m+1}$ is supposed to be outside $\cM^\#$. Moreover, by selecting proper partitions of the space $\F_2^{m+1}$ into disjoint 2-dimensional affine subspaces, we explicitly specify at least $2^{78}$ bent functions in $n=2m=8$ variables  in  $ \cGM_{m+1}$ that are outside $\cM^\#$, which is larger than the cardinality of bent functions in $\cM^\#$. We also provide an explicit design method of bent functions outside $\cM^\#$, based on the concatenation of four bent functions.
\subsection{Bent functions in $\mathcal{GMM}_{\frac{n}{2}+1}$ that belong to $\mathcal{M}^\#$}\label{subsec:bentinMM}
We now demonstrate that the decomposition of $\F_2^{m+1}$ into 2-dimensional affine subspaces must be performed carefully, since otherwise the resulting 
bent functions may belong to the $\cM^\#$ class and are not interesting objects from our perspective.
\begin{prop}\label{prop: trivialpartition}
Let $n=2m$ and let $f \in \cB_n$ be a bent function in the $\cGM_{m+1}$ class, thus satisfying the conditions in Theorem \ref{theo: GMM+1condition},  defined by 
\begin{eqnarray}\label{eq:GMM+1second}
f(x,y)=x\cdot\phi(y) +
h(y), \;\; x \in \F_2^{m-1}, y \in \F_2^{m+1},
\end{eqnarray}
where $\phi
 \colon \F_2^{m+1} \rightarrow \F_2^{m-1}$  and  $h \colon \F_2^{m+1} \rightarrow \F_2$.
Assume that there exists a  non-zero element $v \in \F_2^{m+1}$, such that, for all $z \in \F_2^{m-1}$, we have that $v \in w_z + \phi^{-1}(z)$, for some $w_z \in \phi^{-1}(z)$. Then, the function $f$ is in the (standard) completed Maiorana-McFarland class.
\end{prop}
\begin{proof}
Denote by $V$ the subspace $\langle \F_2^{m-1} \times \{0_{m+1}\}, (0_{m-1},v) \rangle$. Since $v$ is non-zero, the dimension of the subspace $V$ is $m$. We will show that for all $a,b \in V$, we have $D_aD_bf=0$. Firstly, from Eq.~\eqref{eq:GMM+1second} it is easy to see that if $a,b \in \F_2^{m-1} \times \{0_{m+1}\}$, we have $D_aD_bf=0$. Assume now that $a=(0_{m-1},v)$ and $b=(b_1,0_{m+1})$, for some $b_1\in \F_2^{m-1}$. We compute
\begin{align*}
D_aD_bf(x,y)&= D_{(0_{m-1},v)}(D_{(b_1, 0_{m+1})}(x\cdot\phi(y) +
h(y)) \\
&=D_{(0_{m-1},v)}(b_1 \cdot \phi(y)) = b_1 \cdot (\phi(y)+\phi(y + v)).
\end{align*}
We will show that $\phi(y)+\phi(y + v)=0_{m-1}$, for all $y \in \F_2^{m+1}$. Let $z \in \F_2^{m-1}$ be such that $\phi(y)=z$. From Theorem \ref{theo: GMM+1condition}, we know that $\phi^{-1}(z)$ is an affine subspace. Let $w_z \in \phi^{-1}(z)$ be a vector such that $v$ is in $w_z + \phi^{-1}(z)$. Since $\phi^{-1}(z)$ is an affine subspace, we can write it as $w_z + W$  for some $2$-dimensional vector subspace $W \subset \F_2^{m+1}$. It follows that $v \in W$, and consequently $y+v$ also belongs to $\phi^{-1}(z)$, i.e., $\phi(y+v)=z$, so $\phi(y)+\phi(y + v)=0_{m-1}$. It follows that $D_aD_bf(x,y)=0$ in this case as well, so we deduce that for all $a,b \in V$, we have $D_aD_bf=0$, which concludes the proof. 
\end{proof}
\begin{defi}\label{def: properdec}
		A partition of $\F_2^{m+1}$ into 2-dimensional affine subspaces $\{ \phi^{-1}(a) \mid a \in \F_2^{m-1} \}$  will be called {\em non-proper} when it satisfies the conditions given in Proposition \ref{prop: trivialpartition}, otherwise it is termed as {\em proper}.
	\end{defi}
\begin{ex}\label{ex:notsatisfyingMMcond}
	The  partition of $\F_2^5$ w.r.t.\  $\phi^{-1}(a)$ considered in  Example \ref{ex:different_h} is proper, i.e., it does not satisfy the condition given in Proposition \ref{prop: trivialpartition}. Thus, there does not exist a nonzero $v \in \F_2^5$ with the property that for every $a \in \F_2^3$ we have that $v \in w_a + \phi^{-1}(a)$, for some $w_a \in \phi^{-1}(a)$. 
\end{ex}

\begin{cor}\label{cor:insideMM}
Let $n=2m$ and let $f \in \cB_n$ be a bent function in the $\cGM_{m+1}$ class defined by 
\begin{eqnarray}\label{eq:GMM+1second2}
f(x,y)=x\cdot\phi(y) +
h(y), \;\; x \in \F_2^{m-1}, y \in \F_2^{m+1},
\end{eqnarray}
where $\phi
 \colon \F_2^{m+1} \rightarrow \F_2^{m-1}$  and  $h \colon \F_2^{m+1} \rightarrow \F_2$.
Assume that $\{ \phi^{-1}(a) \mid a \in \F_2^{m-1} \}$ is a {\em trivial partition} of $\F_2^{m+1}$ into $2$-dimensional affine subspaces, i.e., assume that there exists a $2$-dimensional linear subspace $W \subset \F_2^{m+1}$ such that  $\{ \phi^{-1}(a) \mid a \in \F_2^{m-1} \} = \{ c + W \mid c \in \F_2^{m+1} \}$, where $\F_2^{m+1} =\bigcup_{c \in Q}(c+W)$ for the set of coset representative $Q$ with $|Q|=2^{m-1}$. Then, the function $f$ has at least $3$ different $\cM$-subspaces of dimension $m$.
\end{cor}
\begin{proof}
We notice that any nonzero element  $v \in W$ satisfies the condition in Proposition \ref{prop: trivialpartition} that $v \in w_z + \phi^{-1}(z)$, by taking $w_z=c$ and $\phi^{-1}(z)=c + W$. In a similar way as in the proof of Proposition \ref{prop: trivialpartition}, it follows that $\langle \F_2^{m-1} \times \{0_{m+1}\}, (0_{m-1},w) \rangle$, for any non-zero $w \in W$, is an $m$-dimensional $\cM$-subspace for $f$, and the result follows. 
\end{proof}

However, as the following example demonstrates, the condition for inclusion in the $\cM^\#$ class given in Proposition \ref{prop: trivialpartition} is sufficient but not necessary.
\begin{ex}\label{ex:necessity}
	Let $\phi\colon\F_2^6\to\F_2^4$ be the mapping defined by $\phi(y)=(\phi_1(y),\phi_2(y),\phi_3(y),\phi_4(y))$, where, for $y=(y_1, y_2, \ldots ,y_6) \in \F_2^6,$
    \begin{equation*}
        \phi(y)= ( y_2,  y_2 y_4 + y_1 y_3 +y_5 + y_2 y_5,  y_1 y_2 + y_2 y_6 + y_6, y_1 + y_2 y_5 + y_1 y_2 ),
    \end{equation*}
	and let $h\in\mathcal{B}_6$ be defined by $h(y) = y_3 y_4 + y_2 y_3 y_4 + y_2 y_3 y_6$. Then, the collection $\left\lbrace \phi^{-1}(a) \mid a \in \F_2^{4} \right\rbrace$ is a partition of $\F_2^{6}$ into $2$-dimensional affine subspaces and the restriction of $h$ on the set $\phi^{-1}(a)$ has odd weight, for every $a\in\F_2^{4}$. Consequently, the function $f \in \mathcal{B}_{10}$ defined by $f(x,y)=x \cdot \phi(y) + h(y)$, for $x \in \F_2^4$ and $y \in \F_2^6$, is bent by Theorem~\ref{theo: GMM+1condition}. By computing the preimages 
    \begin{equation*}
        \begin{split}
            \phi^{-1}(0,0,0,0) = (0,0,0,0,0,0) + \{
            (0,0,0,0,0,0), (0,0,0,1,0,0), (0,0,1,0,0,0), (0,0,1,1,0,0) \}, \\ 
            \phi^{-1}(1,1,1,1) =(1,1,0,1,1,0) + \{ (0,0,0,0,0,0),  (0,0,0,0,0,1),  (0,0,1,1,0,0), (0,0,1,1,0,1) \}, \\
            \phi^{-1}(1,0,0,0) =(0,1,0,0,0,0) + \{ (0,0,0,0,0,0),  (0,0,0,0,0,1),  (0,0,1,0,0,0), (0,0,1,0,0,1) \},
        \end{split}
    \end{equation*}
    we deduce that the mapping $\phi$ does not satisfy the conditions in Proposition~\ref{prop: trivialpartition}. 
    However, by a permutation of variables we obtain
\begin{equation*}
    \begin{split}
        f'(x_1', x_2', x_3', x_4', x_5', y_1', y_2', y_3', y_4', y_5') := f(x_1', y_3', y_4', y_5', x_5', y_1', y_2', x_2', x_3', x_4') = x' \cdot \phi'(y'), \textit{ where }\\ 
        \phi'(y') = ( y_1' , y_2' + y_1' y_2' + y_1' y_3',  y_3' + y_1' y_3' + y_1' y_5', y_1' y_2' + y_4' + y_1' y_4',  y_1' y_4' + y_5' + y_1' y_5' + y_2' y_3'),
    \end{split}
\end{equation*}

for $x'=(x_1', x_2', x_3', x_4', x_5')$ and $y'= (y_1', y_2', y_3', y_4', y_5')$ in $\F_2^5$.
The function $f'$ is clearly in the $\cM$ class, and since $f$ is affine equivalent to $f'$, it follows that $f$ is in the $\cM^\#$ class.
\end{ex}

As originally addressed in \cite{BaumNeuwirth} and for instance later considered in \cite{VanishingFlats2020}, there exist many non-trivial decompositions of vector spaces. For instance, the decomposition of $\F_2^5$ given in Example \ref{ex:different_h} is 8-directional in the sense that for any $i \neq j$ we cannot represent 
$A_i=a_i + A$ and $A_j=a_j +A$ for the same 2-dimensional linear subspace $A$. 
The following remark and example demonstrate that there exist non-trivial partitions, i.e., distinct from the ones described in Corollary~\ref{cor:insideMM}, which also meet the criterion established in Proposition~\ref{prop: trivialpartition}, thereby leading to $f \in \cM^\#$.

\begin{rem}\label{rem:nontrivialpart}
    Note that any trivial partition of $\mathbb{F}_2^{m+1}$ is obviously non-proper. Non-trivial non-proper partitions of $\mathbb{F}_2^{m+1}$ into $2$-dimensional affine subspaces can be obtained as follows.
	Write $\F_2^{m+1}$ as $\F_2^k \times \F_2^t$, for some $k$ and $t$, and let $\{A_i \}_{i \in I}$, $I=\{1, \ldots , 2^{k-1} \}$, be a trivial partition of $\F_2^k$ into $1$-dimensional affine subspaces and $\{B_j \}_{j \in J}$, $J=\{1, \ldots , 2^{t-1} \}$, a non-trivial partition of $\F_2^t$ into $1$-dimensional affine subspaces. 
		Then, the product $\{ A_i \times B_j \}_{i \in I, j \in J}$ of the two partitions is a partition of $\F_2^{m+1}$ into $2$-dimensional affine subspaces which is non-trivial, and such that there is one non-zero vector that is in all the corresponding subspaces of the partition, i.e., the condition of Proposition \ref{prop: trivialpartition} is satisfied. Since any two-element subset of $\F_2^t$ is a $1$-dimensional affine subspace, we can easily obtain non-trivial partitions of $\F_2^t$ into $1$-dimensional affine subspaces (for $t\geq 3$).

        In a similar way, writing $\F_2^{m+1}$ as $\F_2^k \times \F_2^t$ and taking non-trivial partitions of both $\F_2^k$ and $\F_2^t$ into $1$-dimensional affine subspaces and taking the product of the partitions, we obtain proper partitions of $\F_2^{m+1}$ into $2$-dimensional affine subspaces, which can then potentially be used to produce functions outside the $\cM^\#$ class.
\end{rem}

\begin{ex}\label{ex:nontrivialpart}
    Consider the following partition of $\F_2^4$ into four $2$-dimensional affine subspaces:
	\begin{align*}
	\{ (0,0,0,0), (1,0,0,0), (0,1,0,0), (1,1,0,0) \}=(0,0,0,0)+   \langle (1,0,0,0), (0,1,0,0)\rangle, \\
	\{ (0,0,1,0), (1,0,1,0), (0,0,1,1), (1,0,1,1) \}=(0,0,1,0)+   \langle (1,0,0,0), (0,0,0,1)\rangle, \\
	\{ (0,0,0,1), (1,0,0,1), (0,1,0,1), (1,1,0,1) \}= (0,0,0,1)+   \langle (1,0,0,0), (0,1,0,0) \rangle, \\
	\{ (0,1,1,0), (1,1,1,0), (0,1,1,1), (1,1,1,1) \} =(0,1,1,0)+   \langle (1,0,0,0), (0,0,0,1)\rangle .
	\end{align*}
	Since in the partition we have cosets of two different subspaces of $\F_2^4$, namely $\langle (1,0,0,0), (0,1,0,0) \rangle $ and $\langle (1,0,0,0), (0,0,0,1) \rangle$, the partition is not trivial. However, the vector $(1,0,0,0)$ lies in both subspaces, hence the conditions of Proposition \ref{prop: trivialpartition} are again satisfied (for all  $\phi \colon  \F_2^4 \to \F_2^2$ corresponding to this partition) and  $f(x,y)=x \cdot \phi(y) + h(y)$ is in $\cM^\#$, for all $h \in \cB_4$ for which $f$ is bent. In particular, as explained in Remark \ref{rem:nontrivialpart} the partition of $\F_2^4$ presented above is obtained as a product $\{ \{0,1 \} \times B_j \}_{j=1}^4$ of the trivial partition $\{0,1\}$ of $\F_2$ and the following non-trivial partition $\{ B_j \}_{j=1}^4$ of $\F_2^3$
		\begin{align*}
		B_1= \{ (0,0,0),(1,0,0) \}= (0,0,0)+\{ (0,0,0),(1,0,0) \}, \\ B_2= \{ (0,1,0),(0,1,1) \} = (0,1,0)+ \{ (0,0,0),(0,0,1) \}, \\
		B_3= \{ (0,0,1),(1,0,1) \} = (0,0,1)+ \{ (0,0,0),(1,0,0)\} , \\ B_4 =\{ (1,1,0),(1,1,1) \} = (1,1,0)+\{ (0,0,0),(0,0,1) \} . 
		\end{align*}
\end{ex}

\subsection{Specifying $>2^{78}$ different bent  functions  in $\mathcal{GMM}_{\frac{n}{2}+1}$ outside $\mathcal{M}^\#$}\label{subsec:outsideMM}

Despite the possibility of deriving proper partitions using the ideas described in Remark~\ref{rem:nontrivialpart}, a more general approach can be used.
 For small values of $n$, one can construct partitions of $\mathbb{F}_2^n$ into $2^{n-2}$ two-dimensional affine subspaces as follows. Let $G = (V, E)$ be a graph with vertex set $V$ and edge set $E$. The vertices of $G$ represent two-dimensional affine subspaces, and edges are defined as follows: two vertices $V_i$ and $V_j$ are connected by an edge $E_{i,j}$ if and only if $V_i \cap V_j = \varnothing$. A partition of $\mathbb{F}_2^n$ into $2^{n-2}$ two-dimensional affine subspaces corresponds to a clique of size $2^{n-2}$ in $G$. 
 For example, in the case $n = 5$, which is particularly important for the construction of bent functions in 8 variables, one can construct $2^{20}$ such partitions using Wolfram Mathematica in a matter of a few minutes on a single-core computer.
		
 Although it is possible to derive proper partitions using the approach outlined in Remark~\ref{rem:nontrivialpart} or the graph based method above, a straightforward way to obtain partitions of $\mathbb{F}_2^{m+1}$ involves the following procedure:

		\begin{enumerate}
			\item 	Select $A_1= \{0_{m+1}, a_0, b_0, a_0+b_0\}$ which  is a linear subspace of $\F_2^{m+1}$. 
			
			\item 	Then, choose $a_1,b_1,c_1\in \F_2^n\setminus A_1$ and define $A_2=a_1+\{0_{m+1}, b_1+a_1, c_1+a_1, b_1+c_1\}$ which is an  affine subspace of $\F_2^{m+1}$.
			
			\item Continue with  selecting  $a_i,b_i,c_i$ from  $\F_2^{m+1}\setminus  \cup_{j=1}^{i}A_j$ and define $A_{i+1}=\{a_i,b_i,c_i, a_i+b_i+c_i\}$ until all ($2^{m-1}$ in number)  2-dimensional flats $A_i$ are constructed.

		\end{enumerate}  

		We have implemented this procedure and found $4\,960$ different decompositions of $\F_2^5$ into disjoint 2-dimensional affine subspaces, among which there are $3\,785$ proper ones so that the condition in Proposition \ref{prop: trivialpartition} is not satisfied. From each  partition $\Pi=\{A_1,\ldots,A_8\}$, we construct a 4-to-1 mapping $\phi\colon\F_2^5\to\F_2^3$ as $\phi(A_i)=i-1$, for $1\le i\le8$, where every integer $i-1$ is identified with its (3 bit) binary representation. For the constructed mapping $\phi$, we select randomly a Boolean function $h\in\mathcal{B}_5$ such that $f(x, y) = x \cdot \phi(y) + h(y)$ is bent on $\F_2^8$. Using this randomized approach, we obtained $3\,785$ bent functions outside $\mathcal{M}^\#$ class. Below, we give an example of one such bent function (here we use the mapping $\phi\colon\F_2^5\to\F_2^3$ defined by Eq.~\eqref{eq: the first phi}):

		\begin{polynomial}
			f(x,y)=x\cdot\phi(y)+h(y)=x_1 (x_4 + x_4 x_5 x_6 + x_4 x_7 + x_5 x_7 + x_4 x_5 x_7 + x_6 x_7 + x_5 x_6 x_7 + x_4 x_8 + x_5 x_8 + x_6 x_8 + x_5 x_6 x_8 + x_4 x_7 x_8 + x_5 x_7 x_8) + x_2 (x_6 + x_4 x_6 + x_4 x_5 x_6 + x_5 x_7 + x_5 x_6 x_7 + x_4 x_8 + x_5 x_8 + x_4 x_6 x_8 + x_6 x_7 x_8) + x_3 (x_5 + x_4 x_6 + x_5 x_6 + x_4 x_7 + x_4 x_5 x_7 + x_6 x_7 + x_4 x_6 x_7 + x_5 x_6 x_7 + x_4 x_8 + x_4 x_5 x_8 + x_6 x_8 + x_5 x_6 x_8 + x_5 x_7 x_8 + x_6 x_7 x_8) + x_4 x_5 + x_5 x_6 + x_4 x_5 x_6 + x_4 x_7 + x_5 x_7 + x_4 x_6 x_7 + x_4 x_5 x_6 x_7 + x_4 x_8 + x_5 x_8 + x_6 x_8 + x_4 x_6 x_8 + x_5 x_6 x_8 + x_4 x_5 x_6 x_8 + x_7 x_8 + x_4 x_7 x_8 + x_4 x_5 x_7 x_8 + x_6 x_7 x_8,
		\end{polynomial}
	\noindent	where $y:=(x_4,x_5,\ldots,x_8)$. 
	 Using Magma, we have verified that among the constructed $3\,785$ bent functions outside the $\mathcal{M}^\#$ class, there are at least 588 EA-inequivalent ones. We note that this number exceeds the number of EA-equivalence classes of Maiorana-McFarland bent functions, which was recently determined in~\cite{LangevinPolujan2024BFA} and is equal to 325. By computing the orders of the automorphism groups of EA-inequivalent 588 bent functions and applying the orbit-stabilizer theorem, as outlined in \cite[Theorem 2]{LangevinPolujan2024BFA}, we obtain at least $2^{78}$ different bent functions outside $\mathcal{M}^\#$. 

\subsection{Bent functions in $\mathcal{GMM}_{\frac{n}{2}+1}$ outside $\mathcal{M}^\#$ through concatenation}\label{sec:GMM1outsideMM}

We now describe an explicit method of designing  bent functions in $\mathcal{GMM}_{\frac{(2m+2)}{2}+1}$ class outside $\cM^\#$, for $m \geq 3$, using recent results given in \cite{ISIT-IEEE2024}. We first recall one important property called $(P_1)$  introduced in \cite{PPKZ2023} which regards a class of permutations $\pi$ on $\F_2^m$ for which 
\begin{equation}\label{eq: P1} \tag{$P_1$}
D_vD_w\pi\neq0_m \mbox{ for all linearly independent } v,w\in\F_2^m.
\end{equation}
Let $\pi$ be a permutation of $\F_2^m$ with the property $(P_1)$, and let $q \colon \F_2^{2m} \to \F_2$ be the function defined by $q(x,y)=x \cdot \pi(y)$. Then, the only $m$-dimensional $\cM$-subspace of $q$ is $\F_2^m \times \{0_m \}$, by Proposition III.5 in \cite{PPKZ2023}. Let $A$ be any linear permutation of $\F_2^{2m}$ which fixes some $(m-1)$-dimensional subspace of $\F_2^m \times \{0_m \}$, but does not fix the whole subspace $\F_2^m \times \{0_m \}$.
For simplicity, we will assume that 
\begin{equation} \label{eq:matrixA} 
A(x_1, \dots ,x_{m-1},x_{m},y_1, y_2 ,\dots ,y_m)= (x_1, \dots ,x_{m-1},y_{1},x_m, y_2 ,\dots ,y_m).
\end{equation}
Let $g\colon\F_2^{2m} \to \F_2$ be the function defined by $g=q \circ A$. The following result, given in \cite{ISIT-IEEE2024}, will be useful for our purpose.

\begin{theo}\label{th: insideMMgh(5valued)} \cite{ISIT-IEEE2024} Let $h$ and $g$ be two arbitrary bent functions in $\mathcal{B}_n$. Then, the function $f=f_1 \vert \vert f_2 \vert \vert f_3 \vert \vert f_4 \colon \F_2^{n+2} \to \F_2$, where $f_1=f_3=g$ and $f_2=f_4+1=h$ is a bent function in the $\cM^{\#}$ class if and only if the functions $g$ and $h$ have a common $(n/2)$-dimensional $\cM$-subspace, thus $g, h \in \cM^\#$.
\end{theo}
By the construction, the functions $q$ and $g$ share an $(m-1)$-dimensional $\cM$-subspace $\F_2^{m-1} \times \{0_{m+1} \}$, but they do not share an $m$-dimensional $\cM$-subspace.
Consequently, Theorem~\ref{th: insideMMgh(5valued)} implies that the function $f\colon \F_2^{2m+2} \to \F_2$ defined by $f=q \vert \vert g \vert \vert q \vert \vert (g+1)$ is a bent function outside $\cM^\#$. We summarize the above discussion in the following result. 
\begin{theo}\label{th:4concatoutsideMM}
	Let $\pi$ be a permutation of $\F_2^m$ with the property $(P_1)$, and let $q \colon \F_2^{2m} \to \F_2$ be the function defined by $q(x,y)=x \cdot \pi(y)$. Let $A$ be any linear permutation of $\F_2^{2m}$ which fixes some $(m-1)$-dimensional subspace of $\F_2^m \times \{0_m \}$, but does not fix the whole subspace $\F_2^m \times \{0_m \}$.
	 Let $g \colon \F_2^{2m} \to \F_2$ be the function defined by $g=q \circ A$. Then,  $f \colon \F_2^{2m+2} \to \F_2$ defined by $f=q \vert \vert g \vert \vert q \vert \vert (g+1)$ is a bent function outside $\cM^\#$.
\end{theo}

In what follows, we will utilize the structure of $\cM$-subspaces, in particular we will need \cite[Theorem 3.2]{ISIT-IEEE2024}.

	\begin{theo}\cite{ISIT-IEEE2024}\label{th: formofMsubspaces4concatenation}
		Let $f=f_1 \vert \vert f_2 \vert \vert f_3 \vert \vert f_4 \in\mathcal{B}_{n+2}$ be the concatenation of arbitrary Boolean functions $f_1, \ldots ,f_4 \in \mathcal{B}_n$ and let $W$ be a $(k+2)$-dimensional subspace of $\F_2^{n+2}$, for $k \in \{0, \ldots,n \}$.  Then, $W$ is an $\cM$-subspace of $f$ if and only if $W$ has one of the following forms:
		\begin{enumerate}[a)]
			\item \label{item a, th: form} $W= V \times \{(0,0) \}$, where $V \subset \F_2^n$ is a common $(k+2)$-dimensional $\cM$-subspace of $f_1, \ldots ,f_4$. 
			\item \label{item b, th: form} $W= \langle V \times \{(0,0)\}, (a,1,0) \rangle$, where $V$ is a common $(k+1)$-dimensional $\cM$-subspace of $f_1, \ldots ,f_4$, and $a \in \F_2^n$ is such that 
			$$D_vf_1+D_vf_2^a= D_vf_3+D_vf_4^a=0, \text{ for all } v \in V.$$
			\item \label{item c, th: form} $W= \langle V \times \{(0,0)\}, (a,0,1) \rangle$, where $V$ is a common $(k+1)$-dimensional $\cM$-subspace of $f_1, \ldots ,f_4$, and $a \in \F_2^n$ is such that 
			$$D_vf_1+D_vf_3^a= D_vf_2+D_vf_4^a=0, \text{ for all } v \in V.$$
			\item \label{item d, th: form} $W= \langle V \times \{(0,0)\}, (a,1,1) \rangle$, where $V$ is a common $(k+1)$-dimensional $\cM$-subspace of $f_1, \ldots ,f_4$, and $a \in \F_2^n$ is such that 
			$$D_vf_1+D_vf_4^a= D_vf_2+D_vf_3^a=0, \text{ for all } v \in V.$$
			\item \label{item e, th: form}$W= \langle V \times \{(0,0)\}, (a,0,1), (b,1,0) \rangle$, where $V$ is a common $k$-dimensional $\cM$-subspace of $f_1, \ldots ,f_4$, and $a,b \in \F_2^n$ are such that 
			$D_vf_1+D_vf_3^a= D_vf_2+D_vf_4^a=D_vf_1+D_vf_2^b= D_vf_3+D_vf_4^b=0, \text{ for all } v \in V$,  and
			$f_1(x)+f_2(x+b)+f_3(x+a)+f_4(x+a+b)=0$, for all  $x \in \F_2^n$.
		\end{enumerate}
	\end{theo}

Now, we demonstrate that the bent 4-concatenation of the form  $f=q \vert \vert g \vert \vert q \vert \vert (g+1)$ is a bent function in the 
$\mathcal{GMM}_{\frac{2m+2}{2}+1}$ class, even though the functions $g, q$ on $\F_2^{2m}$ are members of the $\cM$ class.
Let $a=0_{2m}$ and $V=\F_2^{m-1} \times \{ 0_{m+1} \}$, then, in our setting, for $v \in V$, we have 
$$D_vf_1+D_vf_3^a=D_vq+D_vq=0 \textnormal{  and  } D_vf_2+D_vf_4^a=D_vg+D_v(g+1)=0.$$ 
Thus,  it follows from Theorem \ref{th: formofMsubspaces4concatenation}, item c), that $W= \langle \F_2^{m-1} \times \{ 0_{m+3} \}, (0_{2m+1},1) \rangle$ is an $\cM$-subspace of $f$ of dimension $m$. Denote by $B$ the linear permutation of $\F_2^{2m+2}$ defined by
$$B(x_1, \dots ,x_{m-1},x_{m}, x_{m+1}, y_1, \dots ,y_{m}, y_{m+1}) =  (x_1, \dots ,x_{m-1},y_{m+1}, x_{m+1}, y_1, \dots ,y_{m}, x_{m}),$$
and let $f' \colon \F_2^{2m+2} \to \F_2$ be the function defined by $f'=f \circ B$. Then, by the construction the subspace $\F_2^{m}\times \{ 0_{m+2} \}$ is an $\cM$ subspace of $f'$, that is $f'$ is in the $\mathcal{GMM}_{\frac{2m+2}{2}+1}$ class by Proposition \ref{prop: genclasschar}. On the other hand, since $f'$ is affine equivalent to $f$ and $f$ is outside $\cM^\#$, $f'$ is also outside $\cM^\#$.
\begin{rem}
	1. It is of interest to represent $f\colon\F_2^{2m+2} \to \F_2$ in Theorem \ref{th:4concatoutsideMM} as $f(x,y)=x \cdot \phi(y) + h(y)$, for a suitably decomposed vector space $\F_2^{2m+2}=\F_2^{m+1 -k} \times \F_2^{m+1+k}$ for $k>0$, and investigate the properties of $\phi \colon  \F_2^{m+1+k} \to \F_2^{m+1-k}$. 
	
	\noindent 2. In Theorem~\ref{th:4concatoutsideMM}, one can employ non-affine permutations $\pi$ of $\F_2^m$ satisfying the $(P_2)$ property (introduced in~\cite{PPKZ2023} and further studied in~\cite{KPPZ2024_JOFC}), which guarantees that the Maiorana-McFarland bent function $q(x,y) = x \cdot \pi(y)$ on $\F_2^m \times \F_2^m$ has the unique canonical $m$-dimensional $\cM$-subspace $\F_2^m \times \{ 0_m \}$; see Proposition III.5 in~\cite{PPKZ2023}.
\end{rem}

\section{Characterizing  the unique $\mathcal{M}$-subspace of maximal dimension}\label{sec:uniqueMandP1}

Although large families of bent functions outside $\cM^\#$ can be designed using proper partitions into 2-dimensional affine subspaces, an explicit algebraic specification of the mapping $\phi\colon\F_2^{m+1} \to \F_2^{m-1}$ has not been addressed. In this section, we provide an analysis of $\cM$-subspaces for functions  $f \in \cGM_{n/2+k}$, with particular focus on the conditions under which such functions admit a unique $\cM$-subspace of maximal dimension $n/2-k$. Then, assuming that $f$ is bent such a characterization immediately ensures its exclusion from $\cM^\#$ whenever $k>0$.  For the characterization of a unique $\cM$-subspace of maximal dimension, we will extend the so-called property \eqref{eq: P1} introduced in \cite{PPKZ2023} and propose a generic method of designing 4-to-1 mappings $\phi \colon \F_2^{m+1} \to \F_2^{m-1}$ satisfying this property. However, ensuring the bentness of $f(x,y)=x \cdot \phi(y) + h(y)$ through the conditions in Theorem \ref{theo: GMM+1condition} appears to be hard, see also Open problem \ref{op:overallconditions}.
	\subsection{Ensuring a unique $\mathcal{M}$-subspace of maximal dimension}
We notice that for $f \in \cGM_{m+1}$ the choice of $h \in \cB_{m+1}$ 
seems to only affect the bent property of $f$ (assuming that $\phi$ partitions $\F_2^{m+1}$ into disjoint 2-dimensional affine subspaces), whereas the mapping $\phi \colon\F_2^{m+1}\to\F_2^{m-1}$ is decisive when the class inclusion of $f$ is of concern (see also Remark \ref{rem:nodependency-on-h}). 

Moreover, for arbitrary $1\leq k \leq n/2-1$, by considering the second-order derivatives of $f(x,y)=x \cdot \phi(y) + h(y)$, with $x \in \F_2^{n/2-k}$, $y \in \F_2^{n/2+k}$ and $\phi \colon\F_2^{m+k}\to\F_2^{m-k}$, we have the following. 
In general, for distinct $a=(a_1,a_2),b=(b_1,b_2) \in \F_2^{n/2-k}\times \F_2^{n/2+k}$  we get:
\begin{equation}\label{eq:secdergeneral}
D_{(a_1, a_2)}D_{(b_1, b_2)}f(x,y)
=x\cdot\left( D_{a_2}D_{b_2}\phi(y)\right)+  a_1\cdot D_{b_2}\phi(y+ a_2)
+ b_1 \cdot D_{a_2}\phi(y+ b_2) + D_{a_2}D_{b_2}h(y) .
\end{equation}

In particular, when $k=1$ (and therefore $f \in \mathcal{GMM}_{n/2+1}$), we notice that for any $a,b \in  V=\F_2^{n/2-1}\times  \{0_{n/2+1}\} $ we have that 
$D_{(a_1, a_2)}D_{(b_1, b_2)}f(x,y)=0$, since $a_2=b_2=0_{n/2+1}$. However, to extend this $\cM$-subspace $V$ on which the second-order derivatives of $f$ vanish with some $u=(u_1,u_2) \in \F_2^{n/2-1}\times \F_2^{n/2+1}$, we necessarily have that $u_2 \neq 0_{n/2+1}$.  This leads to the condition for the class $\cM^\#$ exclusion, since taking $(a_1,a_2) \in V$ and $b=(u_1,u_2)$ in Eq. \eqref{eq:secdergeneral} leads to $a_1\cdot D_{u_2}\phi(y) \neq 0$ (as  $a_2=0_{n/2+1}$ for $a \in V$). 

 We now generalize the property $(P_1)$, introduced in \cite{PPKZ2023} for permutations,
 to cover a larger class of mappings $\phi\colon\F_2^{n/2+k} \to \F_2^{n/2-k}$, which is then proved important in the design of bent functions outside $\cM^\#$. 
\begin{defi}
	Let $\phi\colon\F_2^{m+k} \to \F_2^{m-k}$, for $k=1, \ldots, m-1$, be a function which has the following property:
	\begin{equation}\label{eq: P1ext} \tag{$P_1^*$}
	D_vD_w \phi  \neq 0_{m-k} \mbox{ for all linearly independent } v,w\in\F_2^{m+k}.
	\end{equation}
	Then, we say that $\phi$ satisfies the extended \eqref{eq: P1ext} property. 
\end{defi}

In \cite{PPKZ2023}, the following result was shown.
\begin{theo}\label{theo: unique P1} \cite{PPKZ2023} Let $\pi$ be a permutation of $\F_2^m$ which satisfies \eqref{eq: P1}.
	Define $f\colon\F_2^{m}\times\F_2^{m} \to \F_2$ by $f(x,y)=x \cdot \pi(y) + h(y)$, for all $x,y \in \F_2^{m}$, where $h\colon\F_2^m \to \F_2$ is an arbitrary Boolean function. Then, the following hold:
	\begin{itemize}
		\item[1)] Permutation $\pi$ has no linear structures.
		\item[2)] The vector space $V=\F_2^m \times \{0_m \}$ is the only $m$-dimensional $\mathcal{M}$-subspace of $f$.
	\end{itemize}
\end{theo}

 However, for a mapping $\phi \colon \F_2^{m+k} \rightarrow \F_2^{m-k} $ the situation is somewhat different, especially when relating to the canonical $\cM$-subspace $V=\F_2^{m-k} \times \{0_{m+k} \}$ of $f(x,y)=x \cdot \phi(y) + h(y)$. 

\begin{theo}\label{th:uniqueMsubspace}
	Let $\phi\colon\F_2^{m+k} \rightarrow \F_2^{m-k} $ be a mapping which satisfies \eqref{eq: P1ext}, where $ 0 < k < m$. Define the function $f(x,y)=x \cdot \phi(y) + h(y)$  on $\F_2^{2m}$, where $x \in \F_2^{m-k} $, $y \in \F_2^{m+k}$.  Then, for any $\cM$-subspace $V$ of $f$ it holds that $\dim(V) \leq m-k$. Furthermore, $f$ has the unique $\cM$-subspace $V=\F_2^{m-k} \times \{0_{m+k} \}$ of maximal dimension $m-k$ if and only if there does not exist a nonzero $b_2 \in \F_2^{m+k}$ and an $(m-k-1)$-dimensional subspace $U$ of $\F_2^{m-k}$ such that $u\cdot D_{b_2}\phi(y) = 0$, for all $u \in U$ and $y \in \F_2^{m-k}$. 
    In particular, when $f$ is bent then $f$ is outside $\cM^\#$. 
	\end{theo}
\begin{proof}
    Due to the term $x\cdot\left( D_{a_2}D_{b_2}\phi(y)\right)$ in Eq.~\eqref{eq:secdergeneral},  an $\cM$-subspace $V$ of $f$ cannot have two vectors $(a_1, a_2), (b_1,b_2) \in V$, where $a_1,b_1 \in \F_2^{m-k}$ and $a_2,b_2 \in \F_2^{m+k}$, such that $a_2$ and $b_2$ are linearly independent, since $\phi$ satisfies the property~\eqref{eq: P1ext}.
    Consequently, any $\cM$-subspace $V$ of $f$ with $\dim(V) \geq m-k+1$, must contain a subset of $m-k+1$ linearly independent vectors of the form $\{ (a_1, 0_{m+k}), \ldots , (a_{m-k}, 0_{m+k}) , (b_1, b_2) \}$, where $b_2 \in \F_2^{m+k} \setminus \{ 0_{m+k} \}$. Furthermore, we have $\F_2^{m-k}= \langle a_1, \ldots , a_{m-k} \rangle$, and from Eq.~\eqref{eq:secdergeneral} it follows that $u\cdot D_{b_2}\phi(y) = 0$, for all $u \in \F_2^{m-k}$ and $y \in \F_2^{m-k}$. However, this implies that $D_{b_2}\phi(y) = 0_{m-k}$ for all $y \in \F_2^{m-k}$, hence $D_{b_2'}D_{b_2}\phi = 0_{m-k}$ for any non-zero $b_2' \in \F_2^{m+k} \setminus \{ b_2 \}$, which contradicts the fact that $\phi$ satisfies \eqref{eq: P1ext}. It follows that $\dim(V) \leq m-k$ for any $\cM$-subspace $V$ of $f$.
    
    Similarly, any $\cM$-subspace $V$ of $f$ with $\dim(V) = m-k$, not equal to $\F_2^{m-k} \times \{0_{m+k} \}$, contains a subset of $m-k$ linearly independent vectors of the form $\{ (a_1, 0_{m+k}), \ldots , (a_{m-k-1}, 0_{m+k}) , (b_1, b_2) \}$, where $b_2 \in \F_2^{m+k} \setminus \{ 0_{m+k} \}$. Furthermore, setting $U= \langle a_1, \ldots , a_{m-k-1} \rangle$, from  Eq.~\eqref{eq:secdergeneral} it follows that $u\cdot D_{b_2}\phi(y) = 0$, for all $u \in U$ and $y \in \F_2^{m-k}$. On the other hand, assume that there exist a nonzero $b_2 \in \F_2^{m+k}$ and an $(m-k-1)$-dimensional subspace $U$ of $\F_2^{m-k}$ such that $u\cdot D_{b_2}\phi(y) = 0$, for all $u \in U$ and $y \in \F_2^{m-k}$. Let $\{ a_1, \ldots , a_{m-k-1} \}$ be a basis for $U$. From Eq.~\eqref{eq:secdergeneral} it follows that the subspace $\langle (a_1, 0_{m+k}), \ldots , (a_{m-k-1}, 0_{m+k}) , (0_{m-k}, b_2) \rangle$, is an $(m-k)$-dimensional $\cM$-subspace of $f$, and since $\F_2^{m-k} \times \{0_{m+k} \}$ is also an $\cM$-subspace of $f$, the function $f$ has at least two $(m-k)$-dimensional $\cM$-subspaces.    
    
   In particular, when $f$ is bent, since $k \geq 1$ and the maximal dimension of $\cM$-subspaces of $f$ is $m-k$, it follows that $f$ is outside $\cM^\#$. 
	\end{proof}

It is important to remark that for $k=1$ the bentness of $f(x,y)=x \cdot \phi(y) + h(y)$ is achieved through the conditions given in Theorem \ref{theo: GMM+1condition}.

\begin{rem}\label{rem:sufficiencyP1*}
	Notice that the property \eqref{eq: P1ext} of $\phi$ (through the term $x\cdot\left( D_{a_2}D_{b_2}\phi(y)\right)$ in Eq. \eqref{eq:secdergeneral}) simplifies the design of bent functions in $\cGM_{m+1}$ outside $\cM^\#$ but it is not a necessary condition. 
	For instance, the mapping $\phi$ in  Example \ref{ex:monomialpermNEW} below does not satisfy \eqref{eq: P1ext} even though the bent function $f(x,y)=x \cdot \phi(y) + h(y)$ is outside $\cM^\#$.
\end{rem}

Now, focusing on the special case $k=1$, we first notice that any mapping  $\phi\colon\F_2^{m+2} \to \F_2^m$ can be represented as $\phi=\sigma_1||\sigma_2||\sigma_3 || \sigma_4$, where $\sigma_i$ are mappings on $\F_2^m$. Here, $\sigma_1(y)=\phi(y,0,0), \ldots, \sigma_4(y)=\phi(y,1,1)$, for $y \in \F_2^m$. The difficulty of specifying suitable mappings  $\phi\colon\F_2^{m+2} \to \F_2^m$ of the form $\phi=\sigma_1||\sigma_2||\sigma_3 || \sigma_4$, 
 is then due to the following. In the first place, $\phi$ must be 4-to-1 and provide a proper partition of $\F_2^{m+2}$ (into disjoint 2-dimensional flats that does not  satisfy Proposition \ref{prop: trivialpartition}). Alternatively, to ensure the exclusion from $\cM^\#$, the mapping $\phi$ should satisfy the condition in Theorem \ref{th:uniqueMsubspace}. Of course, to ensure that  $\phi$ is 4-to-1 one can employ $\sigma_i$ that are permutations, which is however not a necessary condition. 
 
\begin{rem}
	For instance, assume  that $\sigma_1, \sigma_2$ are permutations on $\F_2^m$. Then, one can define $\phi\colon\F_2^{m+2}\to\F_2^{m}$ as 
$	\phi(y,y_{m+1},y_{m+2})= \sigma_1(y) +  y_{m+1}(y_{m+2}+1) \sigma_2(y),$
where $y \in \F_2^m, y_{m+1} \in \F_2$, so that $\phi= \sigma_1||\sigma_1||(\sigma_1+\sigma_2)||\sigma_1$.
	The mapping $\phi$ is 4-to-1 if and only if  $\sigma_1+\sigma_2$ is a permutation of $\F_2^m$.  One can  verify that if additionally $\sigma_1$ has the~\eqref{eq: P1} property, then $\phi$ has the~\eqref{eq: P1ext} property. However, $\phi$ does not partition $\F_2^{m+2}$ into 2-dimensional flats  and  therefore is not suitable for our framework. To see this,  notice that writing $\phi=\phi_1||\phi_2||\phi_3||\phi_4$ we have $\phi_1(a)=\phi_2(a)=\phi_4(a)=b$ whereas $\phi_3(a)=\sigma_1(a)+\sigma_2(a)=c$, for some $a,b,c \in \F_2^m$. Since this is valid for any $a \in \F_2^m$ then in general $c \neq b$ and therefore $\phi_3(a')=b$ for some $a \neq a'$. However, $(a,0,0), (a,0,1),(a',1,0),(a,1,1)$ do not build a 2-dimensional flat as their sum is not $0_{m+2}$.
\end{rem} 
The following example illustrates that the property \eqref{eq: P1ext} of $\phi$ is not easily achieved even though $\sigma_i$ might posses the same property. However, the bent function $f(x,y)=x \cdot \phi(y) + h(y)$ is still outside $\cM^\#$, where $\sigma_i$ are monomial permutations on $\F_2^m$ and $\phi=\sigma_1||\sigma_2||\sigma_3 || \sigma_4$. 
	\begin{ex}\label{ex:monomialpermNEW}
		Let $m=4$ and let $\F_{2^4}^*=\langle a \rangle$, where $a^4 + a + 1=0$. Take $d=7$ and let $\alpha_1:=a^6,\alpha_2:=a^{10}$ and $\alpha_3=a$. One can check with a computer algebra system that the mappings $\sigma_i(x)=\alpha_i x^d$  are permutations of $\F_{2^4}$ with the \eqref{eq: P1ext} property, for all $i=1,2,3$. Moreover, $\sigma_i(x)=\alpha_i x^d$ has no components with linear structures since $\operatorname{wt}(d)=3$. Define the mapping $\phi\colon\F_2^6\to \F_2^4$ as follows:
		$$\phi= \sigma_1 || \sigma_2 || \sigma_3 || (\sigma_1 + \sigma_2 + \sigma_3 ) $$
		Since $\sigma_1 + \sigma_2 + \sigma_3 $ is a permutation, we have that $\phi$ is 4-to-1. Its ANF is given by
		$$\phi(y)=(\phi_1(y),\phi_2(y),\phi_3(y),\phi_4(y)),\quad \mbox{where}$$
		\begin{polynomial}
			\phi_1(y)=y_2 + y_2 y_3 + y_4 + y_1 y_2 y_4 + y_3 y_4 + y_1 y_3 y_4 + y_2 y_3 y_4 + y_3 y_5 +
			y_2 y_3 y_5 + y_1 y_4 y_5 + y_2 y_4 y_5 + y_1 y_2 y_4 y_5 + y_1 y_3 y_4 y_5 + y_1 y_6 +
			y_2 y_6 + y_1 y_3 y_6 + y_1 y_2 y_3 y_6 + y_4 y_6 + y_1 y_4 y_6 + y_2 y_4 y_6 +
			y_1 y_2 y_4 y_6 + y_1 y_3 y_4 y_6 + y_2 y_3 y_4 y_6,
		\end{polynomial}
		\begin{polynomial}
			\phi_2(y)=y_1 y_2 + y_3 + y_1 y_3 + y_4 + y_1 y_4 + y_1 y_3 y_4 + y_1 y_5 + y_3 y_5 +
			y_1 y_3 y_5 + y_1 y_2 y_3 y_5 + y_1 y_2 y_4 y_5 + y_3 y_4 y_5 + y_1 y_3 y_4 y_5 +
			y_1 y_6 + y_1 y_3 y_6 + y_2 y_3 y_6 + y_1 y_2 y_3 y_6 + y_1 y_4 y_6 + y_2 y_4 y_6 +
			y_3 y_4 y_6,
		\end{polynomial}
		\begin{polynomial}
			\phi_3(y)=y_1 + y_2 + y_3 + y_1 y_3 + y_2 y_3 + y_1 y_2 y_3 + y_4 + y_2 y_3 y_4 + y_1 y_5 +
			y_1 y_2 y_5 + y_3 y_5 + y_2 y_3 y_5 + y_1 y_2 y_3 y_5 + y_4 y_5 + y_2 y_4 y_5 +
			y_3 y_4 y_5 + y_1 y_3 y_4 y_5 + y_1 y_2 y_6 + y_3 y_6 + y_1 y_3 y_6 + y_4 y_6 +
			y_1 y_4 y_6 + y_1 y_3 y_4 y_6,
		\end{polynomial}
		\begin{polynomial}
			\phi_4(y)=y_1 + y_2 + y_1 y_3 + y_1 y_2 y_3 + y_4 + y_1 y_4 + y_2 y_4 + y_1 y_2 y_4 +
			y_1 y_3 y_4 + y_2 y_3 y_4 + y_1 y_5 + y_2 y_5 + y_1 y_2 y_5 + y_2 y_3 y_5 +
			y_1 y_2 y_3 y_5 + y_1 y_4 y_5 + y_1 y_3 y_4 y_5 + y_2 y_3 y_4 y_5 + y_1 y_6 + y_2 y_6 +
			y_3 y_6 + y_1 y_3 y_6 + y_2 y_3 y_6 + y_1 y_2 y_3 y_6 + y_4 y_6 + y_2 y_3 y_4 y_6.
		\end{polynomial}
		\noindent Note that, the condition $a_1\cdot D_{b_2}\phi(y) \neq 0$ is satisfied for all non-zero $a_1\in\F_2^4$ and $b_2\in\F_2^6$. Moreover, preimages $\phi^{-1}(a)$ form a proper non-trivial partition of $\F_2^6$. The mapping $\phi$ does not have the \eqref{eq: P1ext} property but it has almost \eqref{eq: P1ext} property in the sense that $D_aD_b\phi=0_4$ iff $\langle a,b\rangle=\langle (0, 0, 0, 0, 1, 0), (0, 0, 0, 0, 0, 1) \rangle$. Finally, we confirm that the mapping $f(x,y)=x \cdot \phi(y)$ admits the unique $\cM$-subspace of maximal dimension 4 given as  $V=\F_2^4 \times \{0_6\}$. In this way, one can find Boolean functions $h\in\mathcal{B}_6$ such that $f(x,y)=x \cdot \phi(y) + h(y)$ is a bent function outside the $\mathcal{M}^\#$ class.
\end{ex}

	\subsection{Specifying mappings $\phi\colon\F_2^{m+2} \to \F_2^{m}$ that generate bent functions outside $\mathcal{M}^\#$}
	Example \ref{ex:monomialpermNEW} indicates that an explicit design of a mapping $\phi$ that ensures the exclusion from $\cM^\#$ is possible. Based on the recent characterization in \cite{ISIT-IEEE2024}, we now provide sufficient conditions for $\phi=\sigma_1||\sigma_2||\sigma_3||(\sigma_1+\sigma_2+\sigma_3)$ so that 
	$f(x,y)=x \cdot \phi(y) + h(y)$ is outside $\cM^\#$. For this purpose, we will again utilize Theorem~\ref{th: formofMsubspaces4concatenation}, i.e., \cite[Theorem 3.2]{ISIT-IEEE2024}. 
		
		\begin{theo}\label{th:Th5extension add1}
		Let $x,y\in \F_2^m, y'=(y,y_{m+1},y_{m+1})\in \F_2^{m+2}$.	Let $\phi \colon \F_2^{m+2} \to \F_2^m$ be a 4-to-1 mapping defined as $\phi=\sigma_1||\sigma_2||\sigma_3||\sigma_4$, where $\sigma_i$ and $\sigma_4=\sigma_1+\sigma_2+\sigma_3$ are mappings over $\F_2^m$, $i=1,2,3$. Define $f(x,y)=x \cdot \phi(y') + h(y')$ and represent $f=f_1||f_2||f_3||f_4$, where $f_i \in \cB_m$. Assume that $\sigma_1, \ldots, \sigma_3$ satisfy the following conditions,
		\begin{enumerate}[a)]
			\item The maximum dimension of a common  $\cM$-subspace of $f_1, \ldots ,f_4$ is equal to $m$ and   $V=\F_2^m \times \{0_m\}$  is the only common  $\cM$-subspace of $f_i$,
			
			\item  $\sigma_1 + \sigma_2$  is a permutation over $\F_2^m$  and  $\deg(\sigma_2)=1$,
			
			\item  $\sigma_1(y)\neq \sigma_2(y+u), \sigma_1(y)\neq \sigma_3(y+u), \sigma_2(y)\neq \sigma_3(y+u) $ for any $u\in \F_2^m$.
		
		\end{enumerate}
		 Then, 
			 $f(x,y')=x \cdot \phi(y') + h(y')$ has the unique $\cM$-subspace $\F_2^m \times\{ 0_{m+2}\}$ of maximal dimension $m$. Moreover, if $f$ is bent, implying that $\phi$ is 4-to-1, then $f$ is outside $\cM^\#$. 
		\end{theo}
		\begin{proof}
		Since $\phi=\sigma_1||\sigma_2||\sigma_3||(\sigma_1+\sigma_2+\sigma_3)$, then 
		 $f(x,y')=f_{1}(x,y)||f_{2}(x,y)||f_{3}(x,y)||f_{4}(x,y)$,
		 where $f_{i}(x, y)=x\cdot \sigma_i(y)+ h_i(y)$, for $i=1,\ldots,4$, and additionally $h_1(y)=h(y,0,0),h_2(y)=h(y,0,1),h_3(y)=h(y,1,0), h_4(y)=h(y,1,1)$.
		 
		 From item $a)$ of Theorem \ref{th: formofMsubspaces4concatenation}, we know that  $W=V\times \{(0,0)\}=\F_2^m \times \{0_{m+2}\}$ is an $\cM$-subspace for $f$.
		  Since $\sigma_1(y)\neq \sigma_2(y+u)$  for any $u\in \F_2^m$, then there must exist one $v\in \F_2^m$ such that 
		 $$ D_{(v,0_m)}f_1(x,y)+D_{(v,0_m)}f_2^a(x,y)= v\cdot (\sigma_1(y)+\sigma_2(y+u))\neq 0,$$  for any $a=(u,0_m)\in \F_2^m\times \{0_{m}\}$. 
		 
		  Using item $b)$ of Theorem \ref{th: formofMsubspaces4concatenation}, we have that    $W= \langle V \times \{(0,0)\}, (a,1,0) \rangle$ is not an $\cM$-subspace of $f$, where $V$ is a common $m$-dimensional $\cM$-subspace of $f_1, \ldots ,f_4$. 
		   Since $ \sigma_1(y)\neq \sigma_3(y+u), \sigma_2(y)\neq \sigma_3(y+u) $ for any $u\in \F_2^m$, similarly, there must exist one $v\in \F_2^m$ such that 
		  $$D_{(v,0_m)}f_1+D_{(v,0_m)}f_3^a\neq 0~ (resp.~D_{(v,0_m)}f_2+D_{(v,0_m)}f_3^a\neq 0 ),$$ 
		 for any $a=(u,0_m)\in \F_2^m\times \{0_{m}\}$.  
		 
		    Similarly, considering items $c)-d)$ of Theorem \ref{th: formofMsubspaces4concatenation}, we deduce that    $W= \langle V \times \{(0,0)\}, (a,0,1) \rangle$ and  $W= \langle V \times \{(0,0)\}, (a,1,1) \rangle$ are not  $\cM$-subspace of $f$ due to the conditions $\sigma_1(y) \neq \sigma_3(y+u)$ and $\sigma_2(y) \neq \sigma_3(y+u)$, respectively.

		 Suppose  that $V'$ is a common $(m-1)$-dimensional $\cM$-subspace of $f_1, \ldots , f_4$, thus we consider item $e)$ in Theorem \ref{th: formofMsubspaces4concatenation}.  
		 Set $b=(b^{(1)},b^{(2)})\in\F_2^{m}\times \F_2^m$,  $v'=(v^{(1)},v^{(2)})\in V'\setminus \{0_{2m}\}$.   It is sufficient to prove that $ D_{v'}f_1+D_{v'}f_2^b \neq 0$, where 
		 	\begin{eqnarray} \label{equ.th5.9}
		 				D_{v'}f_1+D_{v'}f_2^b
		 			&=&
		 			x\cdot \left(\sigma_1(y)+\sigma_1(y+v^{(2)})+\sigma_2(y+b^{(2)})+\sigma_2(y+v^{(2)}+b^{(2)})
		 			\right) \nonumber \\
		 			&+&	v^{(1)}\cdot (\sigma_1(y+v^{(2)})+\sigma_2(y+v^{(2)}+b^{(2)})) \\
		 		&	+ &	b^{(1)}\cdot (\sigma_2(y+b^{(2)})+\sigma_2(y+v^{(2)}+b^{(2)}))+D_{v^{(2)}}(h_1(y)+h_2(y+b^{(2)})).\nonumber  
		 		\end{eqnarray}
		 			 	If $v^{(2)}\neq 0_m$, then	since $\deg(\sigma_2)=1$ Eq. \eqref{equ.th5.9}  gives (that the term in the dot product with $x$)
		 		$$ \sigma_1(y)+\sigma_1(y+v^{(2)})+\sigma_2(y)+\sigma_2(y + v^{(2)}) \neq 0_m,$$ 
		 	since $\sigma_1+\sigma_2$ is a permutation, thus $D_vf_1+D_vf_2^b\neq 0$.
		 		  
		 		  	If $v^{(2)}= 0_m$, then   $v^{(1)}\neq 0_m$ since $v=(v^{(1)},v^{(2)})\in V'\setminus \{0_{2m}\}$. 
		 		  	We know
		 		  	$\sigma_1(y)+\sigma_2(y+b^{(2)}) $ is also a permutation 	since $\sigma_1+\sigma_2$ is a permutation and $\deg(\sigma_2)=1$. Further, using the fact that all the components of any permutation are balanced Boolean functions,
		 	$$	D_{v'}f_1+D_{v'}f_2^b=v^{(1)}\cdot (\sigma_1(y)+\sigma_2(y+b^{(2)}))\neq 0.$$
		 

		 	From Theorem \ref{th: formofMsubspaces4concatenation}, we deduce that $f$  has the unique $\cM$-subspace $\F_2^m \times\{ 0_{m+2}\}$ of maximal dimension. Moreover, if $f$ is bent then it is outside $\cM^\#$ by Theorem \ref{th:uniqueMsubspace}. 
		\end{proof}

\subsection{Removing coordinates of permutations} 
In this section, we consider the possibility of deducing the property \eqref{eq: P1ext} after a removal of a suitable number of coordinates of permutations that satisfy \eqref{eq: P1}. 
It is apparent that removing $2k$ coordinate functions from a permutation $\pi \colon \F_2^{n/2 +k} \to \F_2^{n/2 +k}$ results in a mapping $\phi \colon  \F_2^{n/2 +k} \to \F_2^{n/2 -k}$ that is $2^{2k}$-to-1. In the context of  Theorem \ref{th:uniqueMsubspace}, the main question here is whether $\phi$ preserves the property \eqref{eq: P1ext} (thus $D_a D_b \phi \neq 0_{n/2-k}$ for all linearly independent $a, b \in \F_2^{m+k}$) if $\pi$ satisfies the same property.

We first provide  sufficient conditions to  satisfy the property $\eqref{eq: P1ext}$ for certain families of mappings $\phi\colon \F_2^{n/2 +k} \to \F_2^{n/2 -k}$, whose component functions are balanced. 
	\begin{theo}\label{th:mappingswithp1*}
		Let $n \geq 4$ be an even positive integer and $ 0 \leq k  < n/2$.  Let $\phi=(\phi_1,\ldots,\phi_{n/2-k})$ be a mapping from $\F_2^{n/2+k}$ to $\F_2^{n/2-k}$ such that $c\cdot \phi$ is a  balanced function, for any nonzero $c\in \F_2^{n/2-k}$.  For any Boolean function $g \in \cB_{n/2+k}$ , denote $L^{(D_a g)}=\{ v \in \F_2^{n/2+k} \mid D_vD_a g \textnormal{ is constant}\}$, for some nonzero $a \in \F_2^{n/2+k}$.
		 If for any nonzero  $a\in \F_2^{n/2+k}$ we have  $\bigcap\limits_{i=1}^{n/2-k}L^{(D_a \phi_i)}=\{0_{n/2+k},a\}$, 
		then $\phi$ satisfies $(P_1^*)$.
	\end{theo}
	\begin{proof}
		If  $\bigcap\limits_{i=1}^{n/2-k}L^{(D_a \phi_i)}=\{0_{n/2+k},a\}$,
		then for any $b\in \F_2^{n/2+k}\backslash\{0_{n/2+k},a\}$,  there must exist $\ell\in \{1,\ldots,n/2-k\}$ such that $D_aD_b \phi_l \neq constant$, 
		that is, 	we always have 
		$D_aD_b\phi\neq 0_{n/2-k}$. 
		Hence,  $\phi$ satisfies \eqref{eq: P1ext}. 
	\end{proof}
\begin{cor}\label{cor:inheritingP1}
	Let $\Phi=(\Phi_1,\ldots,\Phi_{n/2+k})$ be a permutation on $\F_2^{n/2+k}$ which satisfies the condition of  Theorem \ref{th:mappingswithp1*} so that $\bigcap\limits_{i=1}^{n/2-k}L^{(D_a \Phi_{j_i})}=\{0_{n/2+k},a\}$, for any nonzero $a \in \F_2^{n/2+k}$, where $ \{ j_1,j_2,\ldots, j_{n/2-k}\}\subset \{1,2,\ldots, n/2+k\}$. Then, removing $2k$ coordinates $\Phi_i $  of $\Phi$, where $i\in \{1,2,\ldots, n/2+k\}\setminus \{ j_1,j_2,\ldots, j_{n/2-k}\}$, the derived mapping $\phi\colon \F_2^{n/2+k} \to \F_2^{n/2-k}$ satisfies $(P_1^*)$.
\end{cor} 
\begin{proof}
	Since the condition $\bigcap\limits_{i=1}^{n/2-k}L^{(D_a \Phi_{j_i})}=\{0_{n/2+k},a\}$  is satisfied for $\Phi$, clearly it is valid for $\phi\colon \F_2^{n/2+k} \to \F_2^{n/2-k}$. We notice that since $\Phi$ is a permutation all its component functions are balanced, so are those of $\phi$. 
\end{proof}
The above results essentially enable the use of special class of permutation polynomials $\Phi$ over $\F_2^{n/2+k}$ that satisfy the standard \eqref{eq: P1} property, which may induce  the property \eqref{eq: P1ext} for $\phi\colon\F_2^{n/2+k} \to \F_2^{n/2-k}$. For instance, the monomial permutations over $\F_2^{n/2+k}$ of the form $\Phi(y)=y^d$ were analyzed in terms of the linear structures of their component functions. More precisely, when $\operatorname{wt}(d) >2$ it was shown in \cite{Pascale_Gohar2010} that the component functions of $\Phi$ do not admit linear structures. It can be easily confirmed that removing $2k$ coordinates of $\Phi$ the component functions of $\phi\colon\F_2^{n/2+k} \to \F_2^{n/2-k}$ do not admit linear structures as well.  

However, inheriting the property  \eqref{eq: P1ext} from \eqref{eq: P1} is not  easily achieved. The problem is that we might have the case 
that $D_aD_b \Phi(y)= \Phi(y) + \Phi(y+a) + \Phi(y+b) + \Phi(y+a+b)=(0_{n/2-k}, u),$
where $u \in \F_2^{2k}$ is nonzero. Then, removing the last $2k$ coordinates of $\Phi$, hence specifying $\phi \colon \F_2^{n/2+k} \to \F_2^{n/2-k}$,  we get   $D_aD_b \phi(y)= \phi(y) + \phi(y+a) + \phi(y+b) + \phi(y+a+b)=0_{n/2-k}$ so that \eqref{eq: P1ext} is not satisfied.

\begin{ex}\label{ex:P1extnotbent}
	Consider the APN permutation $\pi(y)=y^{d}$ with $d=14$ on $\F_{2^5}$. Since $\pi$ is APN, it has the \eqref{eq: P1} property and no linear structures, see \cite{PPKZ2023}. Now, we use this permutation to construct a 4-to-1 mapping $\phi\colon\F_2^5\to \F_2^3$, which has the \eqref{eq: P1ext} property and additionally satisfies the following condition in Theorem~\ref{th:uniqueMsubspace}:
	\begin{equation}\label{eq: unique M-subspace}
		\mbox{there does not exist a nonzero } a_2 \in \F_2^{5} \mbox{ s.t. } a_1\cdot D_{a_2}\phi(y) = 0, \mbox{ for all nonzero } a_1 \in \F_2^{3},
	\end{equation}
	which follows from \cite{Pascale_Gohar2010}, as mentioned previously. \newline
 Let $\pi(y)=\pi(y_1,y_2,y_3,y_4,y_5)=(\pi_1(y),\pi_2(y),\pi_3(y),\pi_4(y),\pi_5(y))$ be the ANF of $\pi$. Define the mapping $\phi\colon\F_2^5\to\F_2^3$ by $\phi(y)=(\pi_1(y),\pi_2(y),\pi_3(y))$, for $y\in\F_2^5$. Its ANF is given by:
	$$\phi(y)= \scalebox{0.85}{$\begin{pmatrix}
			y_1 y_2 y_5 + y_1 y_3 y_4 + y_1 y_4 y_5 + y_1 y_5 + y_1 + y_2 y_3 y_5 + y_2 y_4 + y_2 y_5 + y_2 + y_3 y_4 y_5 + y_3 y_4 + y_4 y_5 + y_4 + y_5 \\
			y_1 y_2 y_3 + y_1 y_3 + y_1 y_4 + y_1 y_5 + y_2 y_3 y_5 + y_2 y_4 y_5 + y_2 y_4 + y_2 y_5 + y_3 y_4 y_5 + y_3 y_4 + y_3 y_5 + y_3 + y_4 \\
			y_1 y_2 y_3 + y_1 y_2 y_4 + y_1 y_2 y_5 + y_1 y_2 + y_1 y_3 y_5 + y_1 y_4 + y_2 y_3 y_4 + y_2 y_4 y_5 + y_2 y_4 + y_2 y_5 + y_2 + y_3 + y_4 
		\end{pmatrix}$}^T.$$
	With a computer algebra system, it is possible to check that for the mapping $\phi\colon\F_2^5\to\F_2^3$ the condition in Eq.~\eqref{eq: unique M-subspace} is satisfied. Also,  for all 155 2-dimensional subspaces $\langle a, b\rangle$ of $\F_2^5$ it holds that $D_a D_b \phi \neq 0_5$, so $\phi$ satisfies \eqref{eq: P1ext}, which can also be deduced using Corollary \ref{cor:inheritingP1}. In this way, for any Boolean function $h\in\mathcal{B}_5$, the function $f(x,y)=x \cdot \phi(y) + h(y)$  on $\F_2^{8}$, where $x \in \F_2^{3} $, $y \in \F_2^{5}$, has the unique $\cM$-subspace $\F_2^{3} \times \{0_{5} \}$ of maximal dimension $3$. However, the collection $\left\lbrace \phi^{-1}(a) \mid a \in \F_2^{3} \right\rbrace$ is not a partition of $\F_2^{5}$ into 2-dimensional affine subspaces. It is easy to see that for instance the set $$\phi^{-1}(0_3)=\{ (0, 0, 0, 0, 0), (1, 0, 0, 1, 0), (1, 1, 0, 0, 0), (1, 1, 0, 0, 1) \}$$ is not a flat. Therefore, for any Boolean function $h\in\mathcal{B}_5$, the function $f(x,y)=x \cdot \phi(y) + h(y)$  on $\F_2^{8}$ is not bent.
\end{ex}
Once again the partition into 2-dimensional affine subspaces is decisive in achieving bentness, which leads to the following important research task. 
\begin{op}\label{op:overallconditions}
	Specify 4-to-1 mappings $\phi\colon\F_2^{m+1} \to \F_2^{m-1}$ that satisfy the condition in Theorem~\ref{theo: GMM+1condition} and additionally meet the conditions of Theorem~\ref{th:uniqueMsubspace} (or Theorem~\ref{th:mappingswithp1*}).
\end{op}
	The hardness of the above open problem is further justified by the following result.  

\begin{prop}\label{prop:f_in_MM}
	Let $f\in\mathcal{GMM}_{m+1}$ be a bent function in $2m$ variables. Then, $f\in\mathcal{M}$ if and only if $f$ can be written as $f(x',y')= x' \cdot \phi(y') + h(y')$, where $\phi\colon \F_2^{m+1}\to\F_2^{m-1}$ and $h\in\mathcal{B}_{m+1}$ (in addition to the conditions of Theorem~\ref{theo: GMM+1condition}) satisfy the following two conditions:
	\begin{enumerate}
		\item[1)] The 4-to-1 mapping $\phi$ depends only on the last $m$ variables.
		\item[2)] The Boolean function $h'\in\mathcal{B}_m$ defined by $h'(y):=h(0,y)+h(1,y)$, for all $y\in\F_2^m$, is balanced and extends $\phi(y)=(\phi_1(y),\ldots,\phi_{m-1}(y))$ to a permutation over $\F_2^m$.
	\end{enumerate}
\end{prop}
\begin{proof}
	Let $f\in\mathcal{M}$. Then,  $f(x,y)=x \cdot \pi(y)  + g(y)$, where $x=(x_1,\ldots,x_{m})$ and $y=(y_1,\ldots,y_{m})$ and $\pi$ is a permutation over $\F_2^m$. Let $\pi(y)=(\phi(y),\pi_m(y))=(\pi_1(y),\ldots,\pi_{m-1}(y),\pi_{m}(y))$ and $x=(x',x_m)=((x_1,\ldots,x_{m-1}),x_m),$ $y'=(x_m, y)=(x_m,y_1,\ldots,y_m)$. One can see that
	\begin{equation*}
	\begin{split}
	f(x,y)&=x \cdot \pi(y)  + g(y)\\
	&=x' \cdot \phi(y)  + x_m \pi_m(y) + g(y) \\
	&= x' \cdot \phi(y') +h(y')=f(x',y'),
	\end{split}
	\end{equation*}
	where $h\in\mathcal{B}_{m+1}$ is defined by $h(y')=h(x_m,y)=x_m \pi_m(y) + g(y)$, for $x_m\in\F_2,y\in\F_2^m$. 
	
	Thus, the first claim follows, since $\phi(y')=\phi(x_m, y)=\phi(y)$ holds for all $x_m\in\F_2$, and all $y\in\F_2^m$.  Clearly, the mapping $h'\in\mathcal{B}_m$ defined for all $y\in\F_2^m$  by $h'(y)=h(0,y)+h(1,y)=\pi_m(y)$  is balanced since $\pi_m\in\mathcal{B}_m$ is a coordinate function of the permutation $\pi$ over $\F_2^m$.
	
	Now, assume that a bent function $f\in\mathcal{GMM}_{m+1}$ can be written as $f(x',y')=x' \cdot \phi(y')+h(y')$, where $\phi\colon \F_2^{m+1}\to\F_2^{m-1}$ and $h\in\mathcal{B}_{m+1}$ (in addition to the conditions of Theorem~\ref{theo: GMM+1condition}) satisfy the conditions 1) and 2). Then, 
	\begin{equation*}
	\begin{split}
	f(x',y')&= x' \cdot \phi(y) + x_m (h(0,y)+h(1,y)) + h(0,y) \\
	&= x' \cdot \phi(y) + x_m \pi_m(y) + g(y)= x \cdot \pi(y) + g(y)=f(x,y),\\
	\end{split}
	\end{equation*}
	where  $g(y):=h(0,y)$ and $\pi_m(y):=h(0,y)+h(1,y)$, for $y\in\F_2^m$. Then, the mapping $\pi=(\pi_1,\ldots,\pi_{m-1},\pi_m)$ is a permutation of $\F_2^m$ and $g\in\mathcal{B}_m$, hence $f\in\mathcal{M}$. 
\end{proof}	

	\begin{ex}
		Let $n=2m=8$. Consider the following 8-variable Boolean bent function $f\in\mathcal{GMM}_5$ given by its ANF as follows
		\begin{polynomial}
			(x_1,x_2,x_3,y_1,y_2,y_3,y_4,y_5)\mapsto x_1\phi_1(y_1,\ldots,y_5)+x_2\phi_2(y_1,\ldots,y_5)+x_3\phi_3(y_1,\ldots,y_5)+h(y_1,\ldots,y_5)= x_1 (y_2 y_3 + y_2 y_5 + y_3 y_5 + y_4 y_5 + y_2 y_4 y_5 + y_3 y_4 y_5) +
			x_2 (y_3 + y_4 + y_5 + y_2 y_5 + y_3 y_5 + y_4 y_5) + 
			x_3 (y_2 + y_3 + y_2 y_3 + y_2 y_4 + y_3 y_4 + y_5 + y_4 y_5 + y_2 y_4 y_5 + 
			y_3 y_4 y_5) + 
			y_1 + y_1 y_2 + y_1 y_3 + y_1 y_4 + y_3 y_4 + y_1 y_3 y_4 + y_2 y_5 + y_1 y_2 y_5 + 
			y_1 y_3 y_5 + y_1 y_2 y_3 y_5 + y_2 y_4 y_5 + y_1 y_3 y_4 y_5 + y_2 y_3 y_4 y_5.
		\end{polynomial}
		
		\noindent To see that $f\in\mathcal{M}$, we first observe that none of $\phi_i$ depends on $y_1$. Factoring out $y_1$ in $h(y_1,y_2\ldots,y_5)$ and replacing the variables $y_1=x_4$, $y_2= y_1$, $y_3= y_2$, $y_4= y_3$ and $y_5=y_4$, we get that
		\begin{polynomial}
			f(x,y)= 
			x_1 (y_1 y_2 + y_1 y_4 + y_2 y_4 + y_3 y_4 + y_1 y_3 y_4 + y_2 y_3 y_4) + x_2 (y_2 + y_3 + y_4 + y_1 y_4 + y_2 y_4 + y_3 y_4) + x_3 (y_1 + y_2 + y_1 y_2 + y_1 y_3 + y_2 y_3 + y_4 + y_3 y_4 + y_1 y_3 y_4 + 
			y_2 y_3 y_4) + x_4 (1 + y_1 + y_2 + y_3 + y_2 y_3 + y_1 y_4 + y_2 y_4 + y_1 y_2 y_4 + y_2 y_3 y_4) + y_2 y_3 + y_1 y_4 + y_1 y_3 y_4 + y_1 y_2 y_3 y_4,
		\end{polynomial}
		
		\noindent where $x=(x_1,x_2,x_3,x_4)$ and $y=(y_1,y_2,y_3,y_4)$.
	\end{ex}

\begin{rem}
	Using the same arguments as in~\cite[Eq. (2)]{LangevinPolujan2024BFA}, one can show that any mapping $\phi'\colon\F_2^{m+1}\to\F_2^{m-1}$ that is affine equivalent to a 4-to-1 mapping $\phi\colon\F_2^{m+1}\to\F_2^{m-1}$ satisfying the conditions of Proposition~\ref{prop:f_in_MM}, gives rise to a bent function $f(x',y')=x' \cdot \phi'(y')+h(y')$ on $\F_2^{m-1}\times\F_2^{m+1}$ that belongs to the completed Maiorana-McFarland class $\mathcal{M}^\#$. In this way, the property of $\phi$  ``to depend only on $m$ variables'' can be disguised by the composition of non-degenerate affine mappings $A$ of $\F_2^{m-1}$ and  $B$ of $\F_2^{m+1}$ with $\phi$, i.e., by defining $\phi'$ as $\phi'=A\circ\phi\circ B$. 
\end{rem}

Based on this observation, it is natural to ask the following.

\begin{op}
	Characterize mappings $\phi\colon\F_2^{m+1}\to\F_2^{m-1}$ satisfying the conditions of Theorem~\ref{theo: GMM+1condition} that define bent functions $f\in\mathcal{B}_{2m}$ in $\mathcal{GMM}_{m+1}\setminus\mathcal{M}^\#$.
\end{op}

\section{Conclusions}\label{sec:concl}
In this article, we have provided a full characterization of bent functions of the form $f(x,y) =x \cdot \phi(y) + h(y)$, with $x \in \F_2^{m-1}$ and $y \in \F_2^{m+1}$ that stem from the $\mathcal{GMM}_{m+1}$ class. Their dual bent functions is also specified though  a nice algebraic description is not easily obtained. Proper decompositions of the vector space $\F_2^{m+1}$ into  2-dimensional affine subspaces are of crucial importance, since any  partition 
 that satisfies the condition in Proposition~\ref{prop: trivialpartition}  only leads to bent functions in $\cM^\#$. 
 Although large families of bent functions outside $\cM^\#$ can be designed using proper partitions into 2-dimensional affine subspaces, an explicit algebraic specification of the associated mappings  $\phi\colon\F_2^{m+1} \to \F_2^{m-1}$ remains unclear. In this context, the~\eqref{eq: P1} property—originally introduced in \cite{PPKZ2023} to characterize permutations $\pi$ over $\F_2^m$ satisfying $D_a D_b \pi \neq 0_m$—plays an important role. 
 In this article, we have introduced an extended version of this condition, denoted as~\eqref{eq: P1ext}, which appears to be of similar importance for mappings $\phi\colon \F_2^{m+k} \to \F_2^{m-k}$. Even though the property of being outside $\cM^\#$ is not conditioned on~\eqref{eq: P1ext}, when the class $\cGM_{m+k}$ is considered for $k \geq 1$, it turns out that the property \eqref{eq: P1ext} of $\phi\colon\F_2^{m+k} \to \F_2^{m-k}$ is sufficient to state the exclusion from $\cM^\#$. 
We propose sufficient conditions for these $2^{2k}$-to-1 mappings $\phi$ to satisfy \eqref{eq: P1ext}, but  further investigations are needed (see Open problem~\ref{op:overallconditions}) since $\phi$ additionally must partition $\F_2^{m+k}$ in  a non-trivial (proper) manner (see also Example~\ref{ex:P1extnotbent}). 

\section{Acknowledgments}
Sadmir Kudin is supported in part by the Slovenian Research
Agency (research program P1-0404 and research project J1-4084). Enes Pasalic is supported in part by the Slovenian Research Agency (research program P1-0404 and research projects  J1-1694, J1-4084, J1-2451, N1-0159).
 Fengrong Zhang is supported by the Natural Science Foundation of China under Grant 62372346.
 Haixia Zhao is supported by the Natural Science Foundation of China under Grant 62402132.

\end{document}